\documentclass[a4paper,reqno,twoside]{amsart} 

\usepackage{amssymb}
\usepackage{latexsym}
\usepackage{amsmath}
\usepackage{mathrsfs}
\usepackage{pifont}
\usepackage{enumerate}

\theoremstyle{plain}
\newtheorem{propn}{Proposition}[section]
\newtheorem{thm}[propn]{Theorem}
\newtheorem{lemma}[propn]{Lemma}
\newtheorem{cor}[propn]{Corollary}

\theoremstyle{definition}
\newtheorem*{defn}{Definition}

\theoremstyle{remark}
\newtheorem*{rem}{Remark}
\newtheorem*{rems}{Remarks}

\newcommand{\bm}[1]{\mbox{\boldmath{$#1$}}}

\newcommand{\ve}{\varepsilon}

\newcommand{\Hil}{\mathsf{H}}
\newcommand{\hil}{\mathsf{h}}
\newcommand{\Kil}{\mathsf{K}}
\newcommand{\kil}{\mathsf{k}}
\newcommand{\YVil}{\mathsf{Y}} 
\newcommand{\Vil}{\mathsf{V}}
\newcommand{\Uil}{\mathsf{U}}
\newcommand{\Wil}{\mathsf{W}}
\newcommand{\Xil}{\mathsf{X}}
\newcommand{\Yil}{\mathsf{Y}}

\newcommand{\init}{\mathfrak{h}}
\newcommand{\noise}{\mathsf{k}}

\newcommand{\otM}{\otimes_{\Mat}}
\newcommand{\fullcomp}{\bullet}
\newcommand{\YMotBF}{\Yil \otM B(\FFock)}

\newcommand{\VMotBk}{\Vil \otM B(\kilhat)}
\newcommand{\YMotF}{\Yil \otM |\FFock\ra}
\newcommand{\VMotF}{\Vil \otM |\FFock\ra}

\newcommand{\YMotk}{\Yil \otM |\kilhat\ra}
\newcommand{\VMotk}{\Vil \otM |\kilhat\ra}
\newcommand{\YVMotk}{\YVil \otM |\kilhat\ra} 
\newcommand{\YMotkn}{\Yil \otM |\kilhat^{\ot n}\ra}

\newcommand{\YVMotF}{\YVil \otM |\FFock\ra} 

\newcommand{\Domain}{\mathfrak{D}}

\newcommand{\Clg}{\mathcal{C}}
\newcommand{\Ddense}{\mathcal{D}}
\newcommand{\Ddensestar}{\mathcal{D}_*}
\newcommand{\againDdense}{\mathcal{D}'}
\newcommand{\Exps}{\mathcal{E}}
\newcommand{\FFock}{\mathcal{F}}

\newcommand{\Op}{\mathcal{O}}
\newcommand{\Opdagger}{\mathcal{O}^\ddagger}

\newcommand{\ExpsD}{\mathcal{E}_D}

\newcommand{\Proc}{\mathbb{P}}
\newcommand{\Procdagger}{\mathbb{P}^\ddagger}
\newcommand{\Procwc}{\mathbb{P}_{\mathrm{wc}}}

\newcommand{\Procwr}{\mathbb{P}_{\mathrm{wr}}}
\newcommand{\Procphiwr}{\mathbb{P}_{\phi\mathrm{wr}}}

\newcommand{\Step}{\mathbb{S}}
\newcommand{\StepD}{\mathbb{S}_D}
\newcommand{\againStepD}{\mathbb{S}_{D'}}

\newcommand{\QSC}{\mathbb{QSC}}
\newcommand{\QSCddagger}{\QSC^\ddagger}
\newcommand{\QSCddaggerHc}{\QSCddagger_{\mathrm{Hc}}}
\newcommand{\QSCHc}{\QSC_{\mathrm{Hc}}}
\newcommand{\hED}{\init\ulot\ExpsD}
\newcommand{\againhED}{\init'\ulot\Exps_{D'}}

\newcommand{\counit}{\epsilon}

\newcommand{\Real}{\mathbb{R}}
\newcommand{\Rplus}{\Real_+}
\newcommand{\Comp}{\mathbb{C}}
\newcommand{\Nat}{\mathbb{N}}

\newcommand{\ZZ}{\mathbb{Z}}
\newcommand{\ZZplus}{\mathbb{Z}_+}

\newcommand{\up}{\upsilon}

\newcommand{\Mat}{\mathrm{M}}
\newcommand{\Col}{\mathrm{C}} 

\newcommand{\kB}{\kil{\text{-}B}} 
\newcommand{\cb}{{\text{\tu{cb}}}}
\newcommand{\bd}{{\text{\tu{b}}}}

\newcommand{\wh}{\widehat}
\newcommand{\wt}{\widetilde}
\newcommand{\ol}{\overline}
\newcommand{\ot}{\otimes}
\newcommand{\ulot}{\odot} 
\newcommand{\olot}{\overline{\otimes}}
\newcommand{\spot}{\ot_{\text{\tu{sp}}}}
\newcommand{\uwot}{\olot} 

\newcommand{\la}{\langle}
\newcommand{\ra}{\rangle}

\newcommand{\chat}{\wh{c}}

\newcommand{\Dhat}{\wh{D}}
\newcommand{\fhat}{\wh{f}}
\newcommand{\ghat}{\wh{g}}

\newcommand{\kilhat}{\wh{\kil}}

\newcommand{\tu}{\textup}

\DeclareMathOperator{\Dom}{Dom}
\DeclareMathOperator{\Ran}{Ran}
\DeclareMathOperator{\Lin}{Lin}

\DeclareMathOperator{\id}{id}

\newenvironment{alist}
{

\begin{enumerate}}
{\end{enumerate}}

\newenvironment{rlist}
{

\begin{enumerate}}
{\end{enumerate}}

\numberwithin{equation}{section}
\begin{document}

\title
[Quantum stochastic differential equations]
{On quantum stochastic differential equations}

\author[Lindsay]{J.\ Martin Lindsay}
\address{Department of Mathematics and Statistics, Lancaster University, 
Lancaster LA1 4YF}
\email{j.m.lindsay@lancaster.ac.uk}
\author[Skalski]{Adam G.\ Skalski}
\address{School of Mathematical Sciences, University of Nottingham NG7 2RD}
\footnote{\emph{Permanent address of AGS}.
Department of Mathematics, University of \L\'{o}d\'{z}, ul.
Banacha 22, 90-238 \L\'{o}d\'{z}, Poland.

AGS acknowledges the support of the Polish KBN Research Grant 2P03A 03024 
and EU Research Training Network HPRN-CT-2002-00279}
\email{adam.skalski@maths.nottingham.ac.uk}



\begin{abstract}
Existence and uniqueness theorems for quantum stochastic differential equations
with nontrivial initial conditions are proved for coefficients with
completely bounded columns. Applications are given for the case of
finite-dimensional initial space or, more generally, for coefficients 
satisfying a finite localisability condition. Necessary and sufficient 
conditions are obtained for a conjugate pair of quantum stochastic 
cocycles on a finite-dimensional operator space to strongly satisfy such 
quantum stochastic differential equations. This gives an alternative 
approach to quantum stochastic convolution cocycles on a coalgebra. 
\end{abstract}

\maketitle

\section*{Introduction}
The investigation of quantum stochastic differential equations (QSDE) for 
processes
acting on symmetric Fock spaces dates back to Hudson and Parthasarathy's
founding paper of quantum stochastic calculus
(\cite{RobinPartha}).
As usual in stochastic analysis, these equations are understood as 
integral equations.
By a weak solution is meant a process, consisting of operators (or
mappings), whose matrix elements satisfy certain ordinary integral
equations. Quantum stochastic analysis also harbours a notion of strong 
solution.
The first existence and uniqueness theorems (\cite{RobinPartha}) dealt with
the constant-coefficient operator QSDE with finite-dimensional noise
space; these were soon extended to the mapping QSDE by Evans and Hudson
(\cite{Evans}). Further extensions to the case of infinite-dimensional
noise were obtained in \cite{RobinPartha2}, \cite{MohariSinha} and 
\cite{Fagnola}, and clarified in \cite{Meyer} and \cite{LWptrf}. 
Solutions of such
QSDE's yield quantum stochastic, or Markovian, cocycles (\cite{Accardi}).
The converse is also true under various hypotheses (\cite{HuL}, 
\cite{Bradshaw}); in \cite{LWjfa} it was proved that 
any sufficiently regular cocycle on a $C^*$-algebra satisfies some QSDE
weakly, and moreover if the cocycle is also completely positive and 
contractive, then
it satisfies the equation strongly. In \cite{LWblms} complete boundedness 
of the `columns' of the coefficient was identified as a sufficient 
condition for the 
solution to be strong. (When the noise dimension space is finite 
dimensional boundedness suffices.)
In all the above cases the initial condition for the QSDE
was given by an identity map ampliated to the Fock space.

Parallel to the theory of quantum stochastic cocycles,
Sch\"urmann developed a theory of quantum
L\'evy processes on quantum groups, or more generally $*$-bialgebras,
(see \cite{schu} and references therein).
He showed that each quantum L\'evy process satisfies a QSDE of a certain
type, with initial condition given by the counit of the underlying
$*$-bialgebra (see~\eqref{coalg QSDE} below).
The notion of quantum L\'evy process was recently generalised to quantum 
stochastic \emph{convolution} cocycle on a coalgebra in \cite{LSaihp} 
where it 
was shown that such objects arise as solutions of coalgebraic quantum 
stochastic differential equations. Extension of the results
of that paper to the context of compact quantum groups, or more generally 
operator space coalgebras (\cite{LSqscc2}), was our motivation for 
analysing quantum stochastic differential equations on an operator 
space with nontrivial initial conditions.
Results obtained here have also enabled the development of a dilation 
theory for completely positive convolution cocycles on a $C^*$-bialgebra 
(\cite{Sdilations}).

The aim of this paper is to provide existence and uniqueness results
for a class of quantum stochastic differential equations, under 
natural conditions, together with cocycle characterisation of solutions, 
The crucial role played by complete boundedness (\cite{LWblms}) suggests 
that the main object for consideration as initial space should be an 
operator space. In general operator space theory is very useful for 
describing properties of coefficients, initial conditions and
solutions of our equations (cf.\ \cite{LWCBonOS}).
The main existence theorem is proved for coefficients
with  $\kil$-bounded columns and initial condition given by
a $\kil$-bounded map, where $\kil$ is the `noise dimension space'.
(The term $\kil$-\emph{bounded} means simply bounded if
$\kil$ is finite dimensional and completely bounded otherwise). 
Solutions are expressed in terms of iterated quantum stochastic integrals 
(cf.\ \cite{LWjlms}) and have $\kil$-bounded columns themselves 
(completely bounded columns if the coefficient has cb-columns
and the initial condition is completely bounded).
Due to our choice of test vectors (exponentials of step-functions with 
values in a given dense subspace of the noise dimension space) the results 
are explicitly basis-independent.
As solutions of equations of the type considered are quantum stochastic
cocycles, one may ask which cocycles satisfy a QSDE.
Sufficient conditions for the cocycle to satisfy a QSDE weakly,
established for the case of $C^*$-algebras in \cite{LWjfa}, remain valid
in the coordinate-free, operator space context of this paper.
A new result here, informed by our recent theorem on convolution 
cocycles (\cite{LSaihp}), is the characterisation of cocycles on finite 
dimensional operator spaces which, together with a conjugate process,
satisfy a QSDE strongly --- namely, they are the locally 
H\"older-continuous processes with exponent $1/2$ whose conjugate process 
enjoys the same continuity.

The plan of the paper is as follows. In Section~\ref{preliminaries sec}
the notation is established and basic operator-space theoretic and quantum 
stochastic notions are introduced. There also a concept of finite
localisability is discussed. Weak regularity is shown to be sufficient 
for uniqueness of weak solutions in Section~\ref{regularity sec} (cf.\ 
\cite{LWptrf}). Section~\ref{existence sec} contains the main result on 
the existence of strong solutions of equations on operator spaces and 
elucidates their dependence on initial conditions.
Although in the case of (algebraic) quantum L\'evy processes the initial 
object is a vector space $V$, rather than an operator space, the 
Fundamental Theorem on Coalgebras allows us to effectively work with 
finite-dimensional subspaces and thereby to circumvent the lack of 
analytic structure on $V$ (cf.\ \cite{schu}). 
For this purpose, the version of the existence theorem for finitely localisable
maps relevant for coalgebraic quantum stochastic differential equations
is given in Section~\ref{localisable sec}.
Section~\ref{cocycles sec} begins by recalling known facts on
relations between quantum stochastic cocycles and quantum stochastic
differential equations whose initial condition is given by the identity 
map on a (concrete) operator space. It then gives new necessary and 
sufficient conditions for a conjugate pair of cocycles on a 
finite-dimensional operator space to satisfy a QSDE strongly and ends with 
an application of this result to the infinitesimal generation of quantum 
stochastic convolution cocycles.

\subsection{Notation}
For dense subspaces $E$ and $E'$ of Hilbert spaces $\Hil$ and $\Hil'$,
$\Op (E;\Hil')$ denotes the space of operators $\Hil\to\Hil'$ with domain
$E$ and
$\Opdagger (E,E'):= \{ T\in\Op (E;\Hil'): \Dom T^*\supset E' \}$.
Thus $\Opdagger (E',E)$ is the conjugate space of
$\Opdagger (E,E')$
with conjugation $T\mapsto T^{\dagger}:=T^*|_{E'}$. 
When $\Hil'=\Hil$ we write $\Op(E)$ for $\Op(E;\Hil)$.
We view
$B(\Hil;\Hil')$ as a subspace of $\Opdagger (E,E')$ (via restriction/continuous linear
extension).
For vectors $\zeta\in E$ and $\zeta'\in \Hil'$,
$\omega_{\zeta',\zeta}$ denotes the linear functional on
$\Op (E;\Hil')$ given by
$T\mapsto \la\zeta', T\zeta\ra$, extending a standard notation.
We also use the Dirac-inspired notations
$|E\ra := \{|\zeta\ra : \zeta \in E\}$ and
$\la E| := \{\la\zeta | : \zeta \in E\}$
where $|\zeta\ra\in |\hil\ra := B(\Comp;\hil)$ and
$\la \zeta |\in \la\hil | := B(\hil;\Comp)$ are defined by
$\lambda \mapsto \lambda\zeta$ and $\zeta'\mapsto\la\zeta, \zeta'\ra$
respectively --- inner products (and all sesquilinear maps) here being 
linear in their \emph{second} argument.  

Tensor products of vector spaces, such as dense subspaces of Hilbert 
spaces, are denoted by $\ulot$;
minimal/spatial tensor products of operator spaces by $\spot$; 
and ultraweak tensor products of ultraweakly closed spaces of bounded 
operators  by $\uwot$.
The symbol $\ot$ is used for Hilbert space tensor products and tensor 
products of completely bounded maps between operator spaces; the symbol 
$\ulot$ is also used for the tensor product of unbounded operators, thus 
if $S\in\Op (E;\Hil')$ and $T\in\Op (F;\Kil')$ then
$S\ulot T\in\Op (E\ulot F;\Hil'\ot\Kil')$.
We also need ampliations of bra's and kets: for $\zeta\in\hil$ define
\begin{equation} \label{E-notation}
E^{\zeta}:= I_{\Hil}\ot\la\zeta | \in B(\Hil\ot\hil;\Hil) \text{ and }
E_{\zeta}:= I_{\Hil}\ot |\zeta\ra \in B(\Hil;\Hil\ot\hil),
\end{equation}
where the Hilbert space $\Hil$ is determined by context.  

For a vector-valued function $f$ on $\Rplus$ and subinterval $I$ of
$\Rplus$ $f_I$ denotes the function on $\Rplus$ which agrees with $f$ on
$I$ and vanishes outside $I$. Similarly, for a vector $\xi$, $\xi_I$ is
defined by viewing $\xi$ as a constant function. This extends the standard
indicator function notation. 
The symmetric measure space over the Lebesgue measure space $\Rplus$
(\cite{Guichardet})
is denoted $\Gamma$, with integration denoted
$\int_{\Gamma} \cdots d\sigma$, thus
$\Gamma = \{\sigma\subset\Rplus : \# \sigma < \infty\} =
\bigcup_{n\geq 0} \Gamma^n$ where
$\Gamma^n = \{\sigma\subset\Rplus : \# \sigma =n \}$
and $\emptyset$ is an atom having unit measure.
If $\Rplus$ is replaced by a subinterval $I$ then we write $\Gamma_I$ and
$\Gamma^n_I$,
thus the measure of $\Gamma^n_I$ is $|I|^n/n!$ where $|I|$ is the Lebesgue
measure of $I$.
Finally,
we write $X \subset \subset Y$ to mean that $X$ is a finite subset of $Y$.

\section{Preliminaries} 
\label{preliminaries sec}

\subsection{Quantum stochastics
(\cite{Partha}{\rm ,} \cite{Meyer}{\rm ; we follow} \cite{Llnm})
}
\emph{Fix now, and for the rest of the paper}, a complex Hilbert space
$\kil$ which we refer to as the noise dimension space,
and let $\kilhat$ denote the orthogonal sum $\Comp\oplus\kil$. Whenever $c\in \kil$,
$\chat:=\binom{1}{c}\in \kilhat$; for $E\subset\kil$, 
$\widehat{E}:= \Lin\{\chat:c \in E\}$ and when $g$ is a function with 
values in $\kil$, $\ghat$ denotes the corresponding function with values 
in $\kilhat$ defined by $\ghat(s):= \widehat{g(s)}$.
Let $\FFock$ denote the symmetric Fock space over $L^2(\Rplus;\kil)$.
For any dense subspace $D$ of $\kil$ let
$\Step_D$ denote the linear span of $\{d_{[0,t[}: d\in D, t\in\Rplus\}$
in $L^2(\Rplus;\kil)$ (we always take these right-continuous versions) and
let $\ExpsD$ denote the linear span of $\{\ve(g): g\in\Step_D\}$ in
$\FFock$, where $\ve(g)$ denotes the exponential vector
$\big((n!)^{-\frac{1}{2}}g^{\ot n}\big)_{n\geq 0}$.
The subscript $D$ is dropped when $D=\kil$.
An \emph{exponential domain} is a dense subspace of
$\init\ot\FFock$,
for a Hilbert space $\init$, of the form
$\Domain\ulot\ExpsD$.
We usually drop the tensor symbol and denote
simple tensors such as $v\ot\ve(f)$ by $v\ve(f)$.

For an exponential domain $\Ddense =
\Domain\ulot\ExpsD\subset\init\ot\FFock$ and Hilbert space $\init'$,
$\Proc(\Ddense;\init\ot\FFock)$ denotes the space of (equivalence classes
of) weakly measurable and adapted functions
$X: \Rplus \to \Op (\Ddense;\init'\ot\FFock)$:
\[
t\mapsto \la\xi',X_t\xi\ra \text{ is measurable }
\quad
(\xi'\in\init\ot\FFock, \xi\in\Ddense);
\]
\[
\la u'\ve(g'),X_t u\ve(g)\ra =
\la u'\ve(g'_{[0,t[}),X_t u\ve(g_{[0,t[})\ra
\la u'\ve(g'_{[t,\infty[}), u\ve(g_{[t,\infty[})\ra
\]
($u\in\Domain, g\in\StepD, u'\in\init', g'\in\Step, t\in\Rplus$),
with processes $X$ and $X'$ being identified if, for all $\xi\in\Ddense$,
$X_t\xi = X'_t\xi$ for almost all $t\in\Rplus$.
If $\againDdense$ is an exponential domain in $\init'\ot\FFock$ then
$\Procdagger(\Ddense, \againDdense)$ denotes the space of
$\Opdagger(\Ddense, \againDdense)$-valued processes.
Thus $\Procdagger(\againDdense, \Ddense)$ is the conjugate space of
$\Procdagger(\Ddense, \againDdense)$ with conjugation defined pointwise:
$X_t^{\dagger} = (X_t)^*|_{\againDdense}$.

Let $F\in\Proc(\Domain\ulot\Dhat\ulot\ExpsD;\init'\ot\kilhat\ot\FFock)$
be quantum stochastically integrable (\cite{Llnm}). Then the process
$(X_t = \int_0^tF_s d\Lambda_s)_{t\geq0}\in 
\Proc(\Domain \ulot \ExpsD;\init'\ot\FFock)$
satisfies
\begin{equation}\label{FF1}
\la v'\ve(g'), X_t v\ve(g)\ra =
\int_0^t ds\
\la v' \wh{g'}(s)\ve(g'), F_s v \wh{g}(s)\ve(g)\ra
\end{equation}
\begin{equation}\label{FE}
\| X_t v\ve(g)\|^2 \leq
C(g,t)^2 \int_0^t ds\ \|F_s v \wh{g}(s)\ve(g)\|^2
\end{equation}
($v\in\Domain, g\in\Step, v'\in\init', g'\in\Step, t\in\Rplus$) for a
constant $C(g,t)$ which is independent of $F$ and $v$.
These are known as the Fundamental Formula and Fundamental Estimate
of quantum stochastic calculus. We also need basic estimates for sums of
iterated integrals. Thus let
$L=\big(L_n\in\Op(\Domain\ulot\Dhat^{\ulot n};
\init'\ot\kilhat^{\ot n})\big)_{n\geq 0}$
satisfy the growth condition
\[
\forall_{\gamma\in\Rplus}\forall_{v\in\Domain}\forall_{F\subset\subset\Dhat}
\ \sum_{n\geq 0} \frac{\gamma^n}{\sqrt{n!}}
\max\{\|L_nv\ot\zeta_1\ot\cdots\ot\zeta_n\|: \zeta_1,\ldots,\zeta_n\in F\}
< \infty.
\]
Then the iterated quantum stochastic integrals of the $L_n$ sum to a
process $\big(\Lambda(L)\big)_{t\geq 0}$ satisfying
( for all $v\in\Domain, g\in\Step, v'\in\init', g'\in\Step$)
\begin{equation}\label{FF1full}
\la v'\ve(g'), \Lambda_t(L) v\ve(g)\ra =
e^{\la g,g'\ra}
\int_{\Gamma_{[0,t]}}d\sigma\
\la v'\pi_{\wh{g'}}(\sigma),
L_{\#\sigma} v\pi_{\wh{g}}(\sigma)\ra
\end{equation}

\begin{equation}\label{FEfull}
\|\Lambda_t(L) v\ve(g)\| \leq
\|\ve (g)\|
\sum_{n\geq 0} C(g,T)^n
\Big\{
\int_{\Gamma_{[0,t]}^n}d\sigma\
\|L_n v\pi_{\wh{g}}(\sigma)\|^2\Big\}^{1/2}
\end{equation}

\begin{equation}\label{FDEfull}
\big\|\big[ \Lambda_t(L) - \Lambda_r(L)\big] v\ve(g)\big\| \leq
\|\ve (g)\|
\sum_{n\geq 0} C(g,T)^{n+1}
\Big\{
\int_r^t ds\
\int_{\Gamma_{[0,s]}^n}d\omega\
\|L_n v\pi_{\wh{g}}(\omega)\|^2\Big\}^{1/2},
\end{equation}
for $0\leq r\leq t\leq T$,
where
\[
\pi_{\wh{g}}(\sigma) :=
\wh{g}(s_n)\ot \cdots \ot \wh{g}(s_1)
\text{ for } \sigma = \{s_1 < \cdots < s_n\}\in\Gamma,
\]
with $\pi_{\ghat}(\emptyset):=1$.

\subsection{Forms and maps}
Let $V$ and $V'$ be vector spaces and let $E$ and $E'$ be dense subspaces
of Hilbert spaces $\Hil$ and $\Hil'$.
For any sesquilinear map $\phi$ defined on $E'\times E$ and 
vectors $\zeta'\in E'$ and $\zeta \in E$
we write $\phi^{\zeta'}_{\zeta}$ for the value of $\phi$ at $(\zeta',\zeta)$.
We shall be invoking the following
natural relations:
\begin{align}
SL\big( E',E;L(V;V')\big) &\supset
L\big( E;L(V;V'\ulot|\Hil'\ra)\big) \label{inc 1} \\
&\supset
L\big(  V;V'\ulot\Op(E;\Hil') \big). \label{inc 2}
\end{align}
In case $\Hil$ is finite dimensional the inclusion~\eqref{inc 1}
is an equality. In case $V'$ is finite dimensional the
inclusion~\eqref{inc 2} is an equality. More generally the following
observation is relevant here.

\begin{lemma}
\label{on inc 2 lem}
Let $\chi\in L\big( E;L(V;V'\ulot |\Hil'\ra)\big)$
satisfy the localising property\tu{:}
\[
\forall_{x\in V}\ 
\exists_{V'_1 \textup{ finite dimensional subspace of } V'}\
\forall_{\zeta\in E}\ \
\chi_{|\zeta\ra}(x)\in V'_1\ulot |\Hil'\ra.
\]
Then $\chi\in L\big(  V;V'\ulot\Op(E;\Hil')\big)$.
\end{lemma}
\begin{proof}
Straightforward.
\end{proof}
\begin{defn}
Let $\chi\in L\big(E;L(V;V\ulot|\Hil'\ra)\big)$
for a vector space $V$, pre-Hilbert space $E$ and Hilbert space
$\Hil'$. A subspace $V_1$ of $V$ \emph{localises} $\chi$ if it satisfies
\[
\chi_{|\zeta\ra}(V_1) \subset V_1\ulot |\Hil\ra \quad (\zeta\in E);
\]
$\chi$ is \emph{finitely localisable} if
\[
V = \bigcup
\{ V_1: V_1 \text{ localises } \chi \text{ and } \dim V_1 < \infty \}.
\]
\end{defn}
\begin{rem}
By Lemma~\ref{on inc 2 lem}, if $\chi$ is finitely localisable
then it belongs to $L\big(  V;V\ulot\Op(E;\Hil')\big)$,
and localisation by $V_1$ translates to
\[
\chi (V_1)\subset V_1\ulot \Op(E;\Hil').
\]
\end{rem}

Apart from the case of finite dimensional $V$, the example we have in mind is that of
a coalgebra $\Clg$ with coproduct $\Delta$.
In this context all maps of the form
$\chi = (\id_{\Clg} \ot \varphi ) \circ \Delta$, where
$\varphi \in L\big( \Clg ; \Op (E) \big)$, are finitely localisable. This
follows from the Fundamental Theorem on Coalgebras.

\subsection{Matrix spaces} 
For the general theory of operator spaces and
completely bounded maps we refer to \cite{ERuan} and \cite{ospaces}.
For an operator space $\Yil$ in $B(\Hil;\Hil')$ and Hilbert spaces $\hil$
and  $\hil'$ define
\begin{equation} \label{matrix space defined}
\Yil\otM B(\hil;\hil') :=
\{ T\in B(\Hil\ot\hil; \Hil'\ot\hil') = B(\Hil;\Hil') \uwot B(\hil;\hil'):
\Omega_{\zeta',\zeta}(T)\in\Yil \}
\end{equation}
where $\Omega_{\zeta',\zeta}$ denotes the slice map 
$\id\olot\omega_{\zeta',\zeta}$: 
$T\mapsto E^{\zeta'}TE_{\zeta}$.
For us the relevant cases are
$\Yil\otM B(\hil)$ and
$\Yil\otM |\hil\ra$, referred to respectively as the
$\hil$-matrix space over $\Yil$ and the
$\hil$-column space over $\Yil$.
(Previous notations: $\Mat(\hil;\Yil)_{\bd}$ and $\Col(\hil;\Yil)_{\bd}$.)
Matrix spaces are operator spaces which lie between the spatial tensor
product $\Yil\spot B(\hil;\hil')$ and the ultraweak tensor product
$\ol{\Yil}\uwot B(\hil;\hil')$ ($\ol{\Yil}$ denoting the ultraweak closure 
of $\Yil$). They arise naturally in quantum stochastic
analysis where a topological state space is to be coupled with the
measure-theoretic noise --- if $\Yil$ is a $C^*$-algebra then typically
the inclusion $\Yil\spot B(\hil)\subset \Yil\otM B(\hil)$ is proper and
$\Yil\otM B(\hil)$ is \emph{not} a $C^*$-algebra.
Completely bounded maps between concrete operator spaces lift to
completely bounded maps between corresponding matrix spaces:
for $\phi\in CB(\Yil;\Yil')$ there is a unique map
$\Phi : \Yil\otM B(\hil;\hil')\to\Yil'\otM B(\hil;\hil')$ satisfying
\[
\Omega_{\zeta',\zeta}\circ \Phi = \phi\circ\Omega_{\zeta',\zeta}
\quad
(\zeta\in\hil, \zeta\in\hil').
\]
This map is completely bounded and is denoted 
$\phi\otM \id_{B(\hil;\hil')}$.
A variant on this arises when $\Yil'$ has the form
$\Xil\otM B(\Kil;\Kil')$:
\begin{equation} \label{phi hil}
\phi^{\hil;\hil'} :=
\tau\circ (\phi\otM \id_{B(\hil;\hil')})
\end{equation}
where $\tau$ is the flip on the second and third tensor components, so
that
\[
\phi^{\hil;\hil'}(\Yil\otM B(\hil;\hil'))\subset
\Xil\otM B(\hil;\hil')\otM B(\Kil;\Kil').
\]
When $\hil' = \hil$ we write $\phi^{\hil}$.

\subsection{Tensor-extended composition}
We develop a short-hand notation which will be useful here. 
Let $\Uil, \Vil$ and $\Wil$ be operator spaces and $V$ a vector space. If 
$\phi \in L(V ; \Uil \spot \Vil \spot \Wil)$ and
$\psi \in CB(\Vil ; \Vil')$ then we compose in the obvious way:
\begin{equation} \label{fullcomp}
\psi \fullcomp \phi :=
(\id_{\Uil} \ot \psi \ot \id_{\Wil} ) \circ \phi
\in L(V;\Uil \spot \Vil' \spot \Wil).
\end{equation}
Ambiguity is avoided provided that the context dictates which tensor 
component the second-to-be-applied map $\psi$ should act on. This works 
nicely for matrix-spaces too. Thus if
$\phi \in L(V ; \Yil \otM B(\hil;\hil'))$ and $\psi \in CB(\Yil ; \Yil')$
(or $\psi \in B(\Yil ; \Yil')$ if both $\hil$, $\hil'$ are 
finite-dimensional),
where $\Yil$ and $\Yil'$ are concrete operator spaces, then
\[
\psi \fullcomp \phi := (\psi \otM \id_{B(\hil;\hil')}) \circ 
\phi \in L(V ; \Yil' \otM B(\hil;\hil')) .
\]

The following elementary inequality will be needed in Section 3.

\begin{lemma} \label{psi dot phi}
Let $\psi \in B(\Xil ; \Yil)$ and
$\phi_1 , \ldots , \phi_n \in B\big( \Xil ; \Xil \otM |\Hil\ra \big)$
for concrete operator spaces $\Xil $ and $\Yil$ and finite dimensional Hilbert space $\Hil$.
Then
\[
\| \psi \fullcomp \phi_1 \fullcomp \cdots \fullcomp \phi_n \|
\leq (\dim \Hil)^{n/2} \| \psi \| \, \| \phi_1 \| \cdots \| \phi_n \| .
\]
\end{lemma}

\begin{proof}
Let $(e_i)$ be an orthonormal basis for $\Hil$ and, for a multi-index $\bm{i} = (i_1, \ldots , i_n)$
let $e(\bm{i})$ denote $e_{i_1} \ot \cdots \ot e_{i_n}$. Then, by a `partial Parseval relation'
(recall the `$E$ notation' introduced in \eqref{E-notation})
\[
\| \psi \fullcomp \phi_1 \fullcomp \cdots \fullcomp \phi_n (x) u \|^2
= \sum_{\bm{i}} \big\| E^{e(\bm{i})} ( \psi \fullcomp \phi_1 \fullcomp \cdots \fullcomp \phi_n)
(x) u \big\|^2
\quad (x\in \Xil, u\in\hil)
\]
where $\hil$ is the Hilbert space on which the operators of $\Yil$ act.
The result therefore follows since, for any unit vectors $d_1, \ldots , d_n \in \Hil$,
\begin{align*}
\| E^{d_1 \ot \cdots \ot d_n} \psi \fullcomp \phi_1 \fullcomp  \cdots \fullcomp \phi_n \|
 &= \| \psi \circ E^{d_1} \phi_1 \circ \cdots \circ E^{d_n} \phi_n \| \\
 & \leq \| \psi \| \, \| \phi_1 \| \cdots \| \phi_n \| .
\end{align*}
\end{proof}

The following variant on tensor-extended composition will also be useful.
For $\psi \in L\big(V ; \Op (E \ulot E' ; \Kil \ot \Kil') \big)$ where $V$
is a linear space, $E$ and $E'$ are dense subspaces of Hilbert spaces
$\Hil$ and $\Hil'$ and $\Kil$ and $\Kil'$ are further Hilbert spaces,
\begin{equation} \label{omega dot}  
\omega_{\zeta ,\eta} \fullcomp \psi :=
E^{\zeta} \psi ( \cdot )E_{\eta}, \quad \zeta \in \Kil', \eta \in E' .
\end{equation}
Thus $\omega_{\zeta ,\eta}\fullcomp\psi\in L\big(V;\Op(E ; \Kil)\big)$.

\section{Regularity and uniqueness} \label{regularity sec}

For this section fix a complex vector  space $V$ and exponential domains
$\Ddense = \Domain \ulot \ExpsD$ and
$\againDdense = \Domain' \ulot \Exps_{D'}$ in $\init \ot \FFock$ and
$\init' \ot \FFock$ respectively. A map
$V \to \Proc (\Ddense ; \init' \ot \FFock)$ is called a \emph{process on} $V$. We are
interested in such processes which are \emph{linear} and denote the collection of these by
$\Proc (V : \Ddense ; \init' \ot \FFock)$.
Also define
\[
\Proc^{\ddagger} (V : \Ddense , \againDdense) :=
\big\{k\in \Proc (V : \Ddense ; \init' \ot \FFock): k(V)\subset
\Proc^{\ddagger} (\Ddense ; \againDdense)\big\},
\]
and for such a process $k$  its \emph{conjugate process} 
$k^{\dagger}\in\Proc^{\ddagger} (V^{\dagger} : \againDdense , \Ddense)$
is defined by
$k^{\dagger}_t(x^{\dagger}) = k_t(x)^{\dagger}$.

A process $k$ on $V$ is $(\againDdense , \Ddense)$-\emph{pointwise weakly
continuous} if $s \mapsto ( \omega_{\xi' , \xi} \circ k_s) (x)$ is
continuous for all
$\xi' \in \Ddense'$,
$\xi \in \Ddense$ and $x \in V$; it is $(\againDdense , \Ddense)$-\emph{weakly regular} if,
for some norm on $V$, the following set is bounded
\[
\big\{ \| x \|^{-1} ( \omega_{\xi' ,\xi} \circ k_s ) (x) : x \in V \setminus \{ 0 \} ,
s\in [0,t] \big\}
\]
($\xi' \in \againDdense , \xi \in \Ddense, t \in \Rplus$).
In case
\begin{equation} \label{Dstar}
\Ddense=\Ddensestar:=(\init \ulot \Exps)
\text{ and }
\Ddense'=\Ddensestar':=(\init' \ot \Exps)
\end{equation}
we drop the $(\againDdense , 
\Ddense)$
and refer simply to weakly continuous and weakly regular processes.
 If $V$ already
has a norm  then weak regularity refers to that norm.
We denote the spaces of such processes which are also linear by
$\Procwc (V : \Ddense , \againDdense)$ and
$\Procwr (V : \Ddense, \againDdense)$ respectively.

A weaker notion of regularity tailored to the coefficient of a quantum
stochastic differential equation is
also relevant to the uniqueness question. Thus let $\phi \in SL\big(\wh{D'} , \Dhat ; L(V) \big)$
(sesquilinear maps). For each $R \subset \subset V$, $F \subset \subset D$ and $F' \subset \subset D'$
define the following subspace of $V$
\[
V^\phi_{F', R,F} :=
\Lin \big\{
(\phi^{\zeta'_1}_{\zeta_1} \circ\cdots\circ \phi^{\zeta'_n}_{\zeta_n})(z):
n\in\ZZplus, z \in R, \zeta'_1,\ldots, \zeta'_n\in \wh{F'},
\zeta_1,\ldots , \zeta_n \in \wh{F}
\big\}
\]
(with the convention that an empty product in $L(V)$ equals $\id_V$),
and for $f, f' \in \Step$ write $F'_t $ and $F_t$ for $\Ran f|_{[0,t[}$ and $\Ran f'|_{[0,t[}$
respectively.

\begin{defn}
A process $k:V \to \Proc ( \Ddense ; \init' \ot \FFock)$ is $(\againDdense, \Ddense)$-\emph{weakly
regular locally with respect to} $\phi$ if
$V^\phi_{F'_t , R,F_t}$ has a norm for which the following
is finite:
\begin{equation} \label{Cphi}
C^{k,\phi,t}_{\xi' , R, \xi} = \sup \Big\{ \| z\|^{-1} \big| \omega_{\xi',\xi} \circ
k_s (z) \big|
:
z \in V^\phi_{F'_t , R,F_t} \setminus \{ 0 \}, s \in [0,t[  \Big\}
\end{equation}
($R\subset\subset V,
\xi = v\ve (f)\in\Ddense, \xi' = v'\ve (f') \in \Ddense', t\in\Rplus$).
\end{defn}
\noindent
We shall refer to such norms as \emph{regularity norms} and let
$\Proc_{\phi \mathrm{wr}} (V:\Ddense , \againDdense)$ denote the space of such processes which
are linear.

\begin{propn}
Let $k \in \Procwc (V: \Ddense, \againDdense)$.
\begin{alist}
\item
Let $\phi \in SL\big(\wh{D'} , \Dhat ; L(V)\big)$ and suppose that
$\phi$ satisfies
\[
\dim V^\phi_{F'_t, R, F_t} < \infty
\quad
(R\subset\subset V, f \in \Step_{D}, f'\in\Step_{D'}, t\in\Rplus).
\]
Then $k \in \Proc_{\phi \mathrm{wr}} (V : \Ddense , \againDdense)$.
\item
Suppose that $V$ is a Banach space and $\omega_{\xi' , \xi} \circ k_t$ is bounded
for each $\xi' \in \againDdense , \xi \in \Ddense, t \in \Rplus$.
Then $k \in \Procwr (V: \Ddense , \againDdense)$.
\end{alist}
\end{propn}

\begin{proof}
Let $\xi = u \ve (f) \in \Ddense , \xi' = u' \ve (f') \in \againDdense$ and $t \in \Rplus$.

(a) In this case let $R \subset \subset V$ and consider the $l^1$-norm on
$V^\phi_{F'_t,R,F_t}$
determined by a choice of basis: $\| \sum^d_{i=1} \lambda _i e_i \| := \sum^d_{i=1} | \lambda_i |$.
By linearity
\[
C^{k, \phi,t}_{\xi' , R, \xi} \leq \sup \Big\{ \big| \la \xi' ,k_s (e_i) \xi \ra \big|
: = 0 \leq s \leq t, \ i=1, \ldots , d \Big\} ,
\]
which is finite by weak continuity.

(b) In this case the family of bounded linear functionals
$\{ \omega_{\xi' ,\xi} \circ k_s : 0 \leq s \leq t \}$ is pointwise bounded, by weak continuity,
and so the Banach-Steinhaus Theorem applies.
\end{proof}

\noindent
In particular, if $V$ is finite dimensional then, once equipped with a 
norm, Part (b) applies.
\begin{cor}
If $V$ is finite dimensional then
\[
\Procwc (V:\Ddense, \againDdense) \subset \Procwr (V: \Ddense , \againDdense) .
\]
\end{cor}

\subsection{Quantum stochastic differential equations}
Now let $\phi \in SL\big(\wh{D'}, \Dhat ; L(V) \big)$ and $\kappa \in L(V;W)$
where $W$ is a subspace of
$\Op (\Domain ; \init')$, for example $B(\init;\init')$.
A process $k :V \to \Proc (\Ddense ; \init' \ot \FFock)$
is a $(\againDdense , \Ddense )$-\emph{weak
solution} of the quantum stochastic differential equation
\begin{equation} \label{QSDE}
dk_t = k_t \fullcomp d\Lambda_\phi (t) , \quad k_0 = \iota \circ \kappa 
\end{equation}
(where $\iota$ denotes ampliation 
$\Op (\Domain ; \init') \to \Op (\Ddense ; \init' \ot \FFock )$),
if $k$ is $(\againDdense , \Ddense)$-pointwise weakly continuous and
\begin{align} \label{w soln}
\la \xi', k_t (x) \xi \ra &- \la v' , \kappa (x) v \ra \la \ve(g') ,\ve (g) \ra \notag \\
& \qquad\qquad\qquad\qquad = \int^t_0 ds \Big\la \xi' , k_s
\big(\phi^{\wh{g'}(s)}_{\wh{g}(s)} (x) \big) \xi \Big\ra
\end{align}
($\xi = v \ve (g) \in \Ddense, \xi' = v' \ve (g') \in \againDdense, x \in V, t \in \Rplus$).

\begin{rem}
Suppose that $W$ is a subspace of $\Op^{\ddagger}(\Domain,\Domain')$ and
$\againDdense = \Domain'\ulot\Exps_{D'}$.
If a $(\againDdense , \Ddense )$-weak
solution $k$ of the equation \eqref{QSDE} is 
$\Proc^{\ddag} (\Ddense , \Ddense')$-valued then the conjugate process  
$k^{\dagger}:V^{\dagger}\to\Proc^{\ddagger}(\againDdense,\Ddense)$
is a $( \Ddense, \againDdense )$-weak
solution of the quantum stochastic differential equation \eqref{QSDE} with
$\phi$ and $\kappa$ replaced by
 $\phi^{\dagger} \in SL\big(\wh{D}, \wh{D'} ; L(V^{\dagger}) \big)$ and
 $\kappa^{\dagger} \in L(V^{\dagger};W^{\dagger})$ respectively.
\end{rem}

A process $k \in \Proc (V:\Ddense ; \init' \ot \FFock)$
is a $\Ddense$-\emph{strong solution} of
the quantum stochastic differential equation~\eqref{QSDE} if there is a process
$K \in
\Proc (V: \Domain \ulot\Dhat\ulot\ExpsD; \init'\ot\kilhat \ot \FFock )$
which is pointwise quantum stochastically integrable and satisfies
\begin{equation} \label{strong}
\omega_{\zeta' , \zeta} \fullcomp K_t = k_t \circ \phi^{\zeta'}_{\zeta}
\quad
(\zeta' \in \wh{D'} , \zeta \in \Dhat, t \in \Rplus),
\end{equation}
and
\begin{equation} \label{integral eqn}
k_t (x) = \kappa (x) \ulot I + \int^t_0 K_s (x) \, d \Lambda_s 
\quad
(x \in V, t \in \Rplus).
\end{equation}
In particular strong solutions are (pointwise strongly)
continuous. In view of the First Fundamental Formula~\eqref{FF1},
any $\Ddense$-strong solution is a
$(\againDdense_* , \Ddense)$-weak solution.
Conversely, if $k$ is a $(\againDdense , \Ddense)$-weak solution, with
$\againDdense$ of the form $\Domain' \ulot \Exps_{D'}$, and $K$ is
a pointwise quantum stochastically integrable process satisfying
\eqref{strong} then \eqref{integral eqn} necessarily holds.

Strong solutions will be considered in subsequent sections.
For now let $W= \Op (\Domain ; \init')$.

\begin{thm}\label{linear-unique}
Let $\phi \in SL\big(\wh{D'}, \Dhat ; L(V) \big)$ and $\kappa \in L(V;W)$ 
and let $k$ be a $(\againDdense ,\Ddense)$-weak solution of the quantum 
stochastic differential equation \eqref{QSDE}. 
If $k$ is weakly regular locally with respect to 
$\phi$ and is such that,
for each $R \subset \subset V, v \ve (f) \in \Ddense, v' \ve (f') \in \againDdense,$ $t \in \Rplus$
and $s \in [0,t[$, the map $\phi^{\wh{f'}(s)}_{\fhat (s)}$ is bounded on
$V^\phi_{F'_t ,R,F_t}$
with respect to a corresponding regularity norm, then
\begin{alist}
\item \label{l-u a}
$k$ is linear, so that 
$k\in\Proc_{\phi\tu{wr}}(V:\Ddense,\againDdense)$, and
\item \label{l-u b}
the equation \eqref{QSDE} has no other such solutions.
\end{alist}
\end{thm}

\begin{proof}
Fix $\xi' = u' \ve (f') \in \Ddense' , \xi = u \ve (f) \in \Ddense $ and 
$t \in \Rplus$.

(a) Let $x,y\in V$ and $\lambda\in\Comp$; set $R=\{ x, y, x+\lambda y\}$,
$U=V^\phi_{F'_t , R,F_t}$ with a regularity norm $\|\cdot\|$ and
$C=2C^{k,\phi,t}_{\xi' ,R,\xi}$; and define
\[
\gamma^\lambda_s (z',z)
= \Big\la
\xi' , \big[ k_s (z') + \lambda k_s (z) - k_s (z' +\lambda z) \big] \xi
\Big\ra
\text{ for } z,z' \in U, s \in [0,t] .
\]
By the regularity assumption this satisfies
\[
\big| \gamma^\lambda_s (z',z)\big| \leq C
\big( \| z' \| + | \lambda | \, \| z\| \big).
\]
The linearity of $\kappa$ and each $\phi^{\zeta}_{\zeta'}$ yields the identity
\[
\gamma^\lambda_s (z',z) = \int^s_0 dr \,
\gamma^\lambda_r \Big( \phi^{\wh{f'}(r)}_{\fhat (r)} (z'),
\phi^{\wh{f'}(r)}_{\fhat (r)} (z) \Big) .
\]
Iterating this and using the boundedness assumption gives
\[
\big| \gamma^\lambda_t (x,y) \big| \leq \frac{t^n}{n!} CM^n
\big( \| x\| + | \lambda | \, \| y\| \big), \quad n \in \Nat ,
\]
where
$M = \max \big\{ \| \phi^{\wh{c'}}_{\chat} (z) \| : z \in U, \| z \| \leq 1, c' \in F'_t ,c\in F_t \big\}$.
Thus $\gamma^\lambda_t (x,y)=0$. It follows that $k$ is linear.

(b) Let $\wt{k}$ be another such solution.
For $x\in V$ and $t\in\Rplus$ define
\[
\gamma_s (z) = \big\la \xi', [k_s (z) - \wt{k}_s (z) \big] \xi \big\ra
\qquad \big(z \in V^\phi_{F'_t ,\{ x\} ,F_t} , s \in [0,t]\big).
\]
Then
\[
\big| \gamma_s (z)\big| \leq C
\big( \max\{\|z\|, \| z \|_{\sim} \}\big),
\]
where $C=C^{k,\phi,t}_{\xi' ,\{x\},\xi} + C^{\wt{k},\phi,t}_{\xi' ,\{x\},\xi}$
and $\|\cdot\|$ and $\|\cdot\|_{\sim}$ denote the corresponding regularity 
norms.
Arguing as in (\ref{l-u a}) yields (\ref{l-u b})
\end{proof}

The following two special cases are relevant for the case of coalgebraic (\cite{LSaihp})
and operator space (Section~\ref{existence sec} of this paper) quantum
stochastic differential equations
respectively. The first applies in particular when $V$ is finite dimensional.

\begin{cor}\label{uniqueness cor}
Suppose that $\phi$ satisfies
\[
\dim V^\phi_{F' ,\{ x\} ,F} < \infty \quad
(F' \subset \subset D', x \in V, F \subset \subset D).
\]
Then the quantum stochastic differential equation \eqref{QSDE} has at most
one $(\againDdense ,\Ddense)$-weak
solution. Moreover any such solution is necessarily linear.
\end{cor}

\begin{cor}\label{uniqueness cor2}
Suppose that $V$ is a Banach space and the sesquilinear map $\phi$ is
$B(V)$-valued. Then  the quantum stochastic differential equation
\eqref{QSDE} has at most one linear $(\againDdense ,\Ddense)$-weak
solution $k$ for which each $ \omega_{\xi' ,\xi} \circ k_t$ is bounded
\tu{(}$\xi' \in \againDdense, \xi \in \Ddense, t \in \Rplus$\tu{)}.
\end{cor}

\section{Existence and dependence on initial conditions}
\label{existence sec}

For this section let $\Vil$ be an operator space (with conjugate
operator space $\Vil^{\dagger}$ and conjugation $x\mapsto x^{\dagger}$), 
let $\Yil$ be an operator space in $B(\init; \init')$, 
let $\Ddense = \init \ulot \ExpsD$ and
$\againDdense = \init' \ulot \Exps_{D'}$ for dense subspaces $D$ and $D'$ 
of $\kil$ and recall the notation~\eqref{Dstar}.
Then $\Proc (\Vil \to \Yil : \Ddense , \againDdense )$
denotes the following class of processes on $\Vil$:
\[
\big\{ k \in \Proc ( \Vil : \Ddense ; \init' \ot \FFock ) : \omega_{\ve' , \ve} \fullcomp
k_t (\Vil) \subset \Yil \text{ for all } \ve' \in \Exps_{D'} , \ve \in \ExpsD , t \in \Rplus \big\} .
\]
Recall that $\kil$-bounded means bounded if the noise dimension space 
$\kil$ is finite dimensional and completely bounded otherwise. For 
operator spaces $\Vil$ and $\Wil$, we write
$\kB (\Vil; \Wil)$ for the space of all linear $\kil$-bounded maps acting 
from $\Vil$ to $\Wil$, and give it the operator norm if $\kil$ is 
finite-dimensional and the cb-norm otherwise.

We consider the quantum stochastic differential equation~\eqref{QSDE}
\[
dk_t = k_t \fullcomp d \Lambda_\phi (t), \quad k_0 = \iota \circ \kappa
\]
where
$\phi\in L\big( \Dhat ;\kB\big(\Vil ; CB(\la \kilhat | ;\Vil )\big)\big)
\subset
SL(\kilhat,\Dhat; B(\Vil))$
and $\kappa \in \kB (\Vil ; \Yil)$.
Now ampliation is of bounded operators, so
$\iota(\Yil) \subset \YMotBF$. We say that $\phi$
has `$\kil$-bounded columns' (cf.\ \cite{LWblms}).
Note that $CB(\la\kilhat |;\Vil)=\kil$-$B(\la \kilhat |;\Vil )$
(topological isomorphism).

\begin{thm} \label{existence}
Let
$\phi \in L \big( \Dhat ; \kB \big( \Vil ; CB( \la \kilhat | ; \Vil ) \big) \big)$
and $\kappa \in \kB (\Vil ; \Yil)$.
Then the quantum stochastic differential equation~\eqref{QSDE} has a $\Ddense$-strong solution
$k \in \Proc (\Vil \to \Yil : \Ddense , \againDdense_*)$,
enjoying the following properties
\begin{alist}
\item
\label{ex a}
$k$ has $\kil$-bounded columns\tu{:}
\[
k_{t,|\ve\ra} \in
\kB ( \Vil ; \YMotF) \quad
(t \in \Rplus , \ve\in\ExpsD) .
\]
\item
\label{ex b}
For each $\ve\in\ExpsD$ the map
\[
\Rplus \to \kB \big( \Vil ; \YMotF \big), \quad
s \mapsto k_{s,| \ve \ra}
\]
is locally H\"{o}lder-continuous with exponent $\frac12$.
\item
\label{ex c}
If $\wt{k}$ is a linear $(\againDdense_1 , \Ddense_1)$-weak solution of \eqref{QSDE},
for exponential domains $\againDdense_1$
and $\Ddense_1$ contained in $\againDdense$ and $\Ddense$
respectively, then $\wt{k}$ is a restriction of
$k$: $\wt{k}_t (x) = k_t (x) \big|_{\Ddense_1}$
\ $(x \in \Vil , t \in \Rplus)$.
\item
\label{ex d}
If $\phi$ has cb-columns and $\kappa$ is completely bounded then $k$
has cb-columns and \tu{(}\ref{ex b}\tu{)} holds with
$CB\big( \Vil ; \YMotF \big)$
in place of $\kB\big( \Vil ; \YMotF \big)$.
\end{alist}
\end{thm}

\begin{proof}
Define a process
$k \in \Proc (\Vil \to \Yil : \Ddense , \againDdense)$ as follows:
$k_t = \Lambda_t \circ \up$ where
\[
\up^n \in L \big( \Dhat^{\ulot n} ; \kB (\Vil ; \YMotkn \big)
\subset 
L \big( \Vil ; \Op (\init \ulot \Dhat^{\ulot n} ; 
\init' \ot \kilhat^{\ot n} ) \big)
\quad
(n \in \ZZ)
\]
is defined by
\begin{equation}\label{upsilon}
E^{\zeta'_1 \ot \cdots \ot \zeta'_n} \up^n_{| \zeta_1 \ot \cdots \ot \zeta_n \ra}
= \kappa \circ \phi^{\zeta'_n}_{\zeta_n} \circ \cdots \circ \phi^{\zeta'_1}_{\zeta_1}
\quad
(\zeta_1 , \ldots ,\zeta_n \in \Dhat, \zeta'_1, \ldots , \zeta'_n \in \kilhat).
\end{equation}
Thus, in terms of any concrete
realisation of $\Vil$ in $B(\Hil)$ for a Hilbert space $\Hil$,
\[
\up^n_{| \zeta_1 \ot \cdots \ot \zeta_n \ra} = \tau \circ \left(\kappa \fullcomp
\phi_{| \zeta_n\ra} \fullcomp \cdots \fullcomp \phi_{| \zeta_1 \ra} \right) ,
\]
where $\tau: \YMotkn \to \YMotkn$ denotes the tensor flip reversing the order
of $n$ copies of $\kilhat$.
Therefore, if $\kil$ is finite dimensional then Lemma~\ref{psi dot phi}
implies that
\[
\| \up^n_{| \zeta_1 \ot \cdots \ot \zeta_n \ra} \| \leq \| \kappa \|
\left( \sqrt{\dim \kilhat} \, \max_i \| \phi_{| \zeta_i \ra} \| \right)^n ,
\]
whereas if $\kappa$ is completely bounded and $\phi$ has cb-columns
then
\[
\| \up^n_{| \zeta_1 \ot \cdots \ot \zeta_n \ra} \|_{\cb} \leq \| \kappa \|_{\cb}
\left( \ \max_i \| \phi_{| \zeta_i \ra} \|_{\cb} \right)^n .
\]
It follows from~\eqref{FEfull} and~\eqref{FDEfull} that
$k_{t,| \ve \ra}(\Vil)\subset\YMotF$ and
$k_{t,| \ve \ra}$ is bounded
$\Vil \to \YMotF$
($\ve = \ve(g)  \in \ExpsD, t \in \Rplus$), with
\begin{align*}
&\| k_{t, | \ve \ra} \| \leq
\| \kappa \|' \| \ve \|
 \sum_{n \geq 0} \frac{C^n}{\sqrt{n!}} , \quad \text{ and } \\
& \| k_{t, | \ve \ra} - k_{s, | \ve \ra} \|
 \leq \sqrt{t-s} \| \kappa \|' \| \ve \| C(g,T) \sum_{n \geq 0}
   \frac{C^n}{\sqrt{n!}}
\quad
(0\leq s\leq t\leq T),
\end{align*}
where
$C= C(g,T)\sqrt{C'}
\max \big\{ \| \phi_{| \zeta \ra} \|' :
\zeta \in \Ran \ghat \big|_{[0,T]} \big\}$,
with $\|\cdot \|'$ and $C'$ meaning $\|\cdot\|$ and $\dim\kilhat$ 
respectively, when $\kil$ is finite-dimensional, but $\|\cdot\|_{\cb}$ and 
$1$ otherwise.
We have therefore shown that $k$ satisfies
(\ref{ex a}) and (\ref{ex b}) when $\kil$ is finite dimensional.

Now suppose that $\kappa$ is completely bounded and
$\phi$ has cb-columns. Then, identifying
$M_N\big( \YMotk\big)= \YMotk \otM M_N$
with $M_N(\Yil) \otM |\kilhat\ra = \Yil \otM M_N \otM |\kilhat\ra$ gives
\begin{equation} \label{matrix lifting}
(k_{t, | \ve \ra} )^{(N)} =
\wt{k}_{t, | \ve \ra} \quad
(N \in \Nat , t \in \Rplus , \ve \in \ExpsD),
\end{equation}
where $\wt{k}$ is the process arising from the above construction when
$\kappa$ and $\phi$ are replaced by
$\kappa^{(N)}$ and $\phi^N$,
$\phi^N$ being given by
$(\phi^N)_{| \zeta \ra} = (\phi_{| \zeta \ra})^{(N)}$. It follows
that the above estimates apply with cb-norms on the left-hand side (as 
well as the right).
This completes the proof of (\ref{ex a}), (\ref{ex b}) and (\ref{ex d}).

Recalling~\eqref{FF1full} we next note that
$k$ enjoys the following useful `form representation':
for $\ve = \ve (g) \in \ExpsD$, $\ve' = \ve (g') \in \Exps$ and $t \in \Rplus$,
\begin{equation} \label{form repn}
e^{-\la g',g\ra}
\omega_{\ve' , \ve} \fullcomp k_t
= \int_{\Gamma_{[0,t]}} d \sigma \, \up^{g' ,g}_\sigma
\quad (t\in\Rplus)
\end{equation}
in $B(\Vil;\Yil)$ where
\begin{equation} \label{upsilon g' g}
\up^{g',g}_\sigma
=
\kappa \circ \phi^{\wh{g'}(s_1)}_{\ghat (s_1)} \circ \cdots \circ
\phi^{\wh{g'}(s_n)}_{\ghat (s_n)}
\text{ for }
\sigma = \{ s_1 < \cdots < s_n \}\in\Gamma.
\end{equation}
Therefore
\begin{align*}
\omega_{\ve',\ve}\fullcomp k_t - \la \ve',\ve\ra \kappa
&
=
\la\ve',\ve\ra
\int_{\Gamma_{[0,t]}}d\sigma \, (1- \delta_{\emptyset} (\sigma ))
\up^{g',g}_\sigma \\
&
=
\la\ve',\ve\ra
\int^t_0 ds \int_{\Gamma_{[0,s]}} d \rho\
\up^{g',g}_{\rho\cup \{s\}} \\
&
=
\la\ve',\ve\ra
\int^t_0 ds \int_{\Gamma_{[0,s]}} d \rho\
\up^{g',g}_{\rho} \circ
\phi^{\wh{g'}(s)}_{\ghat (s)} \\
&
=
\int^t_0 ds\
\omega_{\ve',\ve}\fullcomp \big(k_s \circ \phi^{\wh{g'}(s)}_{\ghat (s)}\big),
\end{align*}
so $k_s$ is a ($\Ddense,\againDdense$)-weak solution of~\eqref{QSDE}.

Now define a process
$K \in \Proc \big( \Vil \to \YMotk
 : \init \ulot \Dhat \ulot \ExpsD ,\init' \ulot \wh{D'} \ulot \Exps\big)$ 
by
\[
K_{t, | \zeta \ot\ve \ra} = k_{t,| \ve \ra} \fullcomp
\phi_{| \zeta \ra}
\quad
(t \in \Rplus , \zeta \in \Dhat, \ve \in \ExpsD).
\]
Since it is (pointwise strongly) continuous, by part (b),
$K$ is quantum stochastically integrable.
Moreover, since
\[
E^{\zeta'} K_{t, | \zeta \ot\ve \ra} =
E^{\zeta'} k_{t, | \ve \ra}
\fullcomp \phi_{| \zeta \ra} =
k_{t, | \ve \ra} \circ \phi^{\zeta'}_{\zeta} ,
\]
$K$ also satisfies \eqref{strong}. Therefore $k$ is a
$\Ddense$-strong solution of \eqref{QSDE}.
Part (c) follows from the uniqueness result Corollary~\ref{uniqueness cor2}.
This completes the proof.
\end{proof}

\noindent
\emph{Notation}.
The process uniquely determined by $\kappa$ and $\phi$ in this theorem
will be denoted $k^{\kappa,\phi}$,
extending the established notation $k^\phi$
for the case $\Yil = \Vil$ and $\kappa = \id_{\Vil}$.

\begin{cor}
Let $\phi \in \kB \big(\Vil; CB(T(\kilhat);\Vil)\big)$ and
$\kappa \in \kB(\Vil;\Yil)$. Then \tu{(}for any exponential domains 
$\Ddense$ and $\againDdense$\tu{)} the quantum stochastic differential
equation \eqref{QSDE} has a unique $\Ddense,\againDdense$-weakly regular 
weak solution $k \in \Proc(\Vil\to \Yil:\Ddense,\againDdense)$; it is also 
a $\Ddense$-strong solution. 
\end{cor}

\noindent
Here $T(\kilhat)$ denotes the operator space of trace-class operators on 
$\kilhat$ and we are invoking the natural complete isometry
$CB\big(T(\kilhat);\Vil\big) = 
CB\big(|\kilhat\ra;CB(\la\kilhat |;\Vil)\big)$. 
If $\Vil$ is a concrete operator space then there is a natural completely 
isometric isomorphism between $CB\big(T(\kilhat);\Vil\big)$ and $\VMotBk$, 
so that $\phi$ above may be viewed as a map in 
$\kB \big( \Vil; \VMotBk \big)$.

\begin{cor} \label{conjugate of k phi cor}
Suppose that $\phi$ has a conjugate $\phi^{\dagger}$ in
$L\big( \wh{D'};
\kB (\Vil^{\dagger};CB(\la\kilhat |;\Vil^{\dagger})\big)$.
Then
$k^{\kappa,\phi}\in\Proc^{\ddagger}(\Vil \to \Yil:\Ddense,\againDdense)$ and
$(k^{\kappa,\phi})^{\dagger} =
k^{\kappa^{\dagger},\phi^{\dagger}}$.
\end{cor}
\begin{proof}
In view of the identity
\[
\wt{\up}^{g,g'}_\sigma
= \big(\up^{g',g}_\sigma\big)^{\dagger}
\quad
(g\in\StepD,
g'\in\Step_{D'}, \sigma\in\Gamma),
\]
where $\wt{\up}$ is defined by \eqref{upsilon g' g} with
$\kappa^{\dagger}$ and $\phi^{\dagger}$ in place of
$\kappa$ and $\phi$, this follows from
the form representations~\eqref{form repn} for
$k^{\kappa^{\dagger},\phi^{\dagger}}$ and $k^{\kappa,\phi}$.
\end{proof}

\begin{rems}
(i)
If $U$ is a subspace of $\Vil$ invariant under each of the maps
$\phi^{\zeta'}_\zeta$ ($\zeta'\in\kilhat, \zeta\in\Dhat$) then
$\omega_{\ve', \ve} \fullcomp k_{t}(U)\subset \ol{\kappa(U)}$
for all $\ve'\in\Exps, \ve\in\ExpsD$.

(ii)
The identification~\eqref{matrix lifting} extends as follows. If
$\phi$ has cb-columns and $\kappa$ is completely bounded then
$\hil$-matrix space liftings, of coefficient, initial condition and 
solution, are compatible:
\begin{equation} \label{kappaphihil}
(k^{\kappa,\phi}_t)^{\hil} = k^{\kappa',\phi'}_t
\end{equation}
where $\kappa' = \kappa \otM \id_{B(\hil)}$ and $\phi'$ is determined by
$\phi'_{|\zeta\ra} = (\phi_{|\zeta\ra})^{\hil}$. This follows easily
from the equality
\[ \left(\kappa \fullcomp
\phi_{| \zeta_1\ra} \fullcomp \cdots \fullcomp \phi_{| \zeta_n \ra}\right)^{\hil} =
\kappa' \fullcomp
\phi'_{| \zeta_1\ra} \fullcomp \cdots \fullcomp \phi'_{| \zeta_n \ra}
\]
(in the notation~\eqref{phi hil}and the identity
\[
\Lambda^n(T\ot L) = T\ot \Lambda^n_t(L)
\quad
(T\in B(\hil), n\in\ZZplus,
L\in B(\init;\init')\uwot B(\kilhat^{\ot n})).
\]
\end{rems}

In the next result we consider the case where the operator space $\Vil$
is concrete itself, and so the process $k^{\kappa, \phi}$ may be compared to
the process $k^{\phi}$.

\begin{propn}\label{phi to kappa phi propn}
Let $\kappa$ and $\phi$ be as in Theorem~\ref{existence} and suppose that 
the operator space $\Vil$ is concrete. Then the following hold.
\begin{alist}
\item
\[
\omega_{\ve' , \ve} \fullcomp k^{\kappa , \phi}_t =
\kappa \circ \big(\omega_{\ve' , \ve} \fullcomp k^\phi_t\big)
\quad
(\ve \in \ExpsD, \ve' \in \Exps, t \in \Rplus).
\]
\item
If $\kappa$ is completely bounded then
\[
k^{\kappa , \phi}_{t , | \ve \ra} = \kappa \fullcomp k^\phi_{t, | \ve \ra} \quad
(t \in \Rplus , \ve \in \ExpsD) .
\]
\item
If $\kappa$ is completely bounded and the process $k^\phi$ is completely bounded then
$k^{\kappa , \phi}$ is the completely bounded process given by
\[
k^{\kappa , \phi}_t = \kappa \fullcomp k^\phi_t \quad (t \in \Rplus ).
\]
\end{alist}
\end{propn}
\begin{proof}
(a) follows easily from~\eqref{form repn}; (b) and (c) are simple consequences of (a).
\end{proof}

\begin{rems}
Since the process $k^{\kappa,\phi}$ depends linearly on $\kappa$, the proposition implies that
it also depends continuously on its initial condition --- in various senses, depending on
the regularity of the initial condition and process $k^\phi$.

If $\Vil = \Yil$ and the initial condition commutes with the coefficient 
operator,
in the sense that
$\kappa\fullcomp\phi_{|\zeta\ra}
= \phi_{|\zeta\ra}\circ\kappa$
($\zeta\in\Dhat$),
then
$\kappa\fullcomp\phi^{\fullcomp n}_{|\eta\ra}
= \phi^{\fullcomp n}_{|\eta\ra}\circ\kappa$
  ($n\in\ZZplus, \eta\in\Dhat^{\ulot n}$)
and so
\[
k_t^{\kappa,\phi} = k_t^{\phi}\circ\kappa
\quad (t\in\Rplus).
\]

Injectivity of the quantum stochastic operation $\Lambda$
(\cite{LWjlms}, Proposition 2.3) implies that
\[
k^{\kappa,\phi}= k^{\kappa',\phi'}
\text{ if and only if }
\kappa = \kappa' \text{ and }
\kappa\fullcomp\phi_{|\zeta\ra}
= \kappa'\fullcomp\phi'_{|\zeta\ra}
\quad
(\zeta\in\Dhat).
\]
\end{rems}

\section{Localisable equations} \label{localisable sec}

In this section we consider the case where the source space is a vector
space on which the coefficient map of the  quantum stochastic
differential equation is finitely localisable. Thus let $V$
be a complex vector space, let $D$ be a dense subspace
of the noise dimensions space $\noise$ and consider
our quantum stochastic differential equation~\eqref{QSDE}
\[
dk_t = k_t \fullcomp d \Lambda_\phi (t) , \quad k_0 = \iota \circ \kappa ,
\]
where $\phi \in L \big( \Dhat ; L(V ; V \ulot | \kilhat \ra ) \big) $.
We consider two cases.
Recall that if $\phi$ is finitely localisable then it necessarily belongs to
$L(V;V\ulot\Op(\Dhat))$; also recall the notation~\eqref{Dstar}.

\begin{thm} \label{local existence}
Let $\phi \in L(V;V\ulot\Op(\Dhat))$
be finitely localisable
and let $\kappa \in L(V;\Yil )$, where $\Yil$ is an operator space in $B(\init ; \init')$. Set
$\Ddense = \init \ulot \ExpsD$.
Then there is a process
$k \in \Proc (V \to \Yil : \Ddense , \againDdense_*)$,
which is a $\Ddense$-strong
solution of \eqref{QSDE} and enjoys the following further properties:
\begin{alist}
\item
\label{lex a}
$k$ is $L(V;\Yil\ulot\Op(\ExpsD))$-valued.
\item
\label{lex b}
The map
$s \mapsto k_{s, | \ve \ra} (x)$ is locally
H\"{o}lder-continuous
$\Rplus \to \Yil \spot | \FFock \ra$ with exponent $\frac12$
\tu{(}$x \in V, \ve\in\ExpsD$\tu{)}.
\item
\label{lex c}
If $\wt{k}$ is a $(\againDdense_1 , \Ddense_1)$-weak solution of
\eqref{QSDE}, where
$\Ddense_1$ and $\againDdense_1$ are exponential domains contained in
$\Ddense$ and $\againDdense_*$
respectively, then $\wt{k}$ is a restriction of
$k$: $\wt{k}_t (x) = k_t (x) |_{\Ddense_1}$.
\item
\label{lex d}
For any subspace $V_1$ localising $\phi$,
$k_{t} (V_1) \subset \kappa(V_1) \ulot \Op(\ExpsD)$
$(t\in\Rplus)$.
\end{alist}
\end{thm}

\begin{proof}
Consider a finite dimensional subspace $V_1$ of $V$ which localises $\phi$ and let $\kappa_1$
and $\phi_1$ be the restrictions of $\kappa$ and $\phi$ to $V_1$. By endowing $V_1$ with operator
space structure $\kappa_1$ becomes completely bounded and $\phi_1$ enjoys completely bounded columns.
Theorem~\ref{existence} therefore permits us to define a process
$k^1 \in \Proc (V_1 \to \Yil :\Ddense , \againDdense )$ by $k^1 =
k^{\kappa_1 , \phi_1}$. Now
suppose that $k^2 \in \Proc (V_2 \to \Yil : \Ddense , \againDdense )$ is the process arising in
this way from another finite dimensional subspace $V_2$ localising $\phi$.
Then the finite
dimensional subspace $V_3 := V_1 \cap V_2$ also localises $\phi$ and so gives rise to a third
process $k^3 \in \Proc (V_3 \to \Yil : \Ddense , \againDdense)$. By the uniqueness part of
Theorem~\ref{existence} it follows that $k^3$ agrees with both $k^1$ and $k^2$ on $V_3$. The following
prescription therefore gives a consistent definition of a process
$k \in \Proc (V \to \Yil : \Ddense , \againDdense )$: let
$k_t (x) =
k^{\kappa_1, \phi_1}_t (x) $ where $\kappa_1$ and $\phi_1$
are the restrictions of $\kappa$ and $\phi$ to any finite dimensional
subspace of $V$ containing
$x$ which localises $\phi$. That $k$ is a $\Ddense$-strong solution of \eqref{QSDE}
satisfying properties (\ref{lex a})-(\ref{lex d}) now
follows easily from Theorem~\ref{existence} and the subsequent remark. Observe that
(\ref{lex d}) implies that for each $s \geq 0$ and $\ve \in \ExpsD$ the map $k_{s,|\ve\ra}$
takes values in $\Yil \ulot |\FFock\ra$.
\end{proof}

\begin{rem}
Clearly the following weaker localisable property suffices: for all $x \in 
V$ and $F \subset \subset D$
there is a finite dimensional subspace $V_1$ of $V$ containing $x$ such that
$\phi_{| \zeta \ra} (V_1) \subset V_1 \ulot | \kilhat \ra$ for all $\zeta 
\in \wh{F}$;
conclusion~(\ref{lex d}) is then modified accordingly.
\end{rem}

\noindent
\emph{Notation}. We again use the notation $k^{\kappa,\phi}$ for the process
obtained in the above theorem.

As before,
\[
k^{\kappa,\phi}= k^{\kappa',\phi'}
\text{ if and only if }
\kappa = \kappa' \text{ and }
\kappa\fullcomp\phi
= \kappa'\fullcomp\phi'.
\]

\begin{cor} \label{algebraic conjugate}
Suppose that $\phi\in L(V;V\ulot\Op^{\ddagger}(\Dhat,\wh{D'}))$
for some dense subspace $D'$ of $\kil$. Then
$k^{\kappa,\phi}\in\Proc^{\ddagger}(V:\Ddense, \againDdense)$ where
$\againDdense = \init'\ulot\Exps_{D'}$ and
$(k^{\kappa,\phi})^{\dagger} =
k^{\kappa^{\dagger},\phi^{\dagger}}$.
\end{cor}

We next give a variant of the above existence theorem. Note that the 
definition of
$\Proc (V \to \Yil : \Ddense , \againDdense)$ extends in an obvious way if $\Yil$ is replaced
by $W=\Op (\Domain ; \init')$ and $\Ddense$ by $\Domain \ulot \ExpsD$.

\begin{thm} \label{more local existence}
Let $\phi \in L(V;V\ulot\Op(\Dhat))$
be finitely localisable, let $\kappa \in L(V; W)$
and set $\Ddense = \Domain \ulot \ExpsD$.
Then the conclusions of Theorem~\ref{local existence} hold
with
$\Yil$ replaced by $W$ and
\tu{(}\ref{lex a}\tu{)}, \tu{(}\ref{lex b}\tu{)} and
\tu{(}\ref{lex d}\tu{)}
replaced by
\begin{alist}
\item[\tu{(a)}$'$]  \label{mle a}
$s \mapsto k^{\kappa , \phi}_s (x) \xi$ is locally H\"{o}lder-continuous
$\Rplus \to \init' \ot \FFock$ with exponent $\frac12$, for all $x \in V$ and $\xi \in \Ddense$.
\end{alist}
\end{thm}

\begin{proof}
For $u \in \Domain$, Theorem~\ref{local existence} applies, with 
$\Yil = | \init' \ra$, to the quantum stochastic differential equation
\[
dk_t =
k_t \fullcomp d \Lambda_\phi (t), \quad k_0 = \iota \circ \kappa_{| u\ra};
\]
Let $l^u\in\Proc\big(V\to |\init'\ra:\Exps_D,\init'\ulot\Exps\big)$ 
be its $\ExpsD$-strong solution.
For $u,v \in \Domain$ and $\lambda \in \Comp$,
if $g \in \ExpsD$ and $\xi' = v'\ve(g') \in \againDdense$ then the maps
$\gamma_s : V \to \Comp$ $(s \in \Rplus)$ given by
\[
\gamma_s (x) = \Big\la \xi' , \big[l^u_s (x) + \lambda l^v_s (x) -
l^{(u+\lambda v)}_s (x) \big]
\ve (g) \Big\ra
\]
satisfy
\[
\gamma_t (x) =
\int^t_0 ds \, \gamma_s \big( \phi^{\wh{g'}(s)}_{\ghat (s)} (x) \big)
\quad (x \in V , t \in \Rplus) .
\]
In view of finite localisability, iteration shows that $\gamma$ is identically zero. If
follows that
\[
k^{\kappa , \phi}_t (x) u \ve (g) := l^u_t (x) \ve (g)  \quad
(x \in V, u \in \Domain, g \in \Step_D , t \in \Rplus) ,
\]
defines a process
$k^{\kappa ,\phi} \in \Proc (V \to W : \Ddense , \againDdense )$ which is
a $\Ddense$-strong solution of \eqref{QSDE}; it is clear that it satisfies
(a)$'$ and (\ref{lex c}) too.
\end{proof}

\section{Quantum stochastic cocycles} \label{cocycles sec}

In this section we give a new result on the infinitesimal generation of
quantum stochastic cocycles (cf.~\cite{LWjfa}).
At the end we describe how the result may be applied to quantum stochastic
convolution cocycles on a coalgebra (\cite{LSaihp}).
Fix an operator space $\YVil$ in $B(\init ; \init')$ and exponential
domains $\Ddense = \hED$ and $\againDdense = \againhED$.

The following notations for a process
$k \in \Proc (\YVil \to \YVil : \Ddense , \againDdense )$ prove useful:
\begin{equation}\label{kfg}
k^{g',g}_t :=
e^{-\la g'_{[0,t[} ,g_{[0,t[} \ra}
\omega_{\ve (g'_{[0,t[}), \ve(g_{[0,t[})} \fullcomp k_t
\end{equation}
($g' \in \againStepD , g \in \StepD, t \in \Rplus$) and
\begin{equation}\label{kcd}
k^{c',c}_t :=
k^{c'_{[0,t[},c_{[0,t[}} \quad (c'\in D', c\in D).
\end{equation}
Thus $k^{g',g}_t\in L(\YVil)$ and the process is called \emph{initial 
space
bounded} if each map $k^{g',g}_t$ is bounded (cf.\,the condition of having
bounded columns).

\begin{defn}
A process $k \in \Proc (\YVil \to \YVil : \Ddense, \againDdense )$
is a $(\againDdense, \Ddense)$-\emph{weak quantum stochastic cocycle}
on $\YVil$ if it satisfies
\begin{equation}\label{cocycle reln}
k^{g',g}_{r+t} = k^{g',g}_r \circ k^{S^*_r g', S^*_r g}_t
\end{equation}
for all $g' \in \againStepD$, $r,t\in\Rplus$ and $g \in \StepD$,
where $(S_t)_{t\geq0}$ is the (isometric) right-shift semigroup on
$L^2(\Rplus;\noise)$.
\end{defn}
\noindent Let $\QSC(\YVil:\Ddense,\againDdense)$ denote the collection of 
these.
Also define
\[
\QSC^\ddagger(\YVil:\Ddense,\againDdense) =
\QSC(\YVil:\Ddense,\againDdense) \cap
\Proc^\ddagger(\YVil \to \YVil:\Ddense,\againDdense);
\]
if $k$ is in this class then
$k^{\dagger\,g,g'}_t = (k_t^{g',g})^\dagger$
and it is easily seen that the conjugate process
$k^\dagger$ is a cocycle on $\YVil^\dagger$.

In case the process has
cb-columns (each map $x\mapsto k_{t, |\ve \ra}(x)$ is completely bounded
$\YVil\to\VMotF$)
the cocycle relation is equivalent to
\[
k_{r+t, | \ve (g_{[0,r+t[}) \ra} = k_{r,|\ve (g_{[0,t[})\ra}
\fullcomp k_{t, | \ve(S^*_r g_{[r,r+t[})\ra} ;
\]
in case the process itself is completely bounded it simplifies further, to
the more recognisable cocycle property:
\[
k_{r+t} = k_r \fullcomp \sigma_r \fullcomp k_t
\]
were $(\sigma_r)_{r \geq 0}$ is the CCR flow of index $\noise$
(\cite{Arveson}).

\begin{lemma}\label{sgp decomp lem}
Let $k \in \Proc (\YVil \to \YVil : \Ddense, \againDdense)$ and define
$P^{c',c} := (k^{c',c}_t)_{t\geq 0}$ \tu{(}$c',c\in\kil$\tu{)}. Then the 
following are equivalent\tu{:}
\begin{rlist}
\item
$k\in\QSC (\YVil :\Ddense,\againDdense)$.
\item
For all $c'\in D'$ and $c\in D$, $P^{c',c}$ is a one-parameter semigroup in 
$L(\YVil)$
and, for all $g' \in \againStepD , g \in \StepD$ and $t \in \Rplus$,
$k^{g',g}_t = l^{g',g}_t$ where
\begin{equation}\label{lfg}
l^{g',g}_t =
P^{g'(t_0),g(t_0)}_{t_1 -t_0} \cdots P^{g'(t_n),g(t_n)}_{t_{n+1} -t_{n}}
\end{equation}
with $n\in\ZZplus$, $t_0=0$, $t_{n+1}=t$ and $\{t_1 < \cdots < t_n\}$
being precisely the \tu{(}possibly empty\tu{)} union of the sets of points
of discontinuity of $g'$ and $g$ in $]0,t[$.
\item
For all $g' \in \againStepD , g \in \StepD$ and $t \in \Rplus$,
\begin{equation} \label{sgp decomp}
k^{g',g}_t =
P^{g'(t_0),g(t_0)}_{t_1 -t_0} \cdots P^{g'(t_n),g(t_n)}_{t_{n+1} -t_{n}}
\end{equation}
whenever $n\in\ZZplus$ and $\{ 0=t_0 \leq \cdots \leq t_{n+1} =t\}$
includes all the discontinuities of $g'_{[0,t[}$ and $g_{[0,t[}$.
\end{rlist}
\end{lemma}
\noindent
\begin{proof}
Straightforward, see~\cite{LWjfa}.
\end{proof}

The one-parameter semigroups
$\{ P^{c',c} : c'\in D', c\in D\}$
in $L(\YVil)$ are referred to as the \emph{associated semigroups}
of $k$, $P^{0,0}$ as its \emph{Markov semigroup}
and~\eqref{sgp decomp} as its \emph{semigroup decomposition}.
If $k$ is initial space bounded and each semigroup is norm continuous
$\Rplus\to B(\YVil)$ then the cocycle is called \emph{Markov-regular}. When
the cocycle is contractive, norm continuity of any of the associated
semigroups (such as its Markov semigroup) implies Markov-regularity
(\cite{LWjfa}, Proposition 5.4). In view of the semigroup decomposition,
Markov-regular cocycles are necessarily both weakly regular and weakly
continuous processes.

Now consider the quantum stochastic differential equation~\eqref{QSDE}
where $\kappa = \id_{\YVil}$:

\begin{equation}\label{id QSDE}
dk_t = k_t \fullcomp d\Lambda_\phi(t), \quad k_0 = \iota.
\end{equation}
The following result is a coordinate-free counterpart to Proposition 5.2
of~\cite{LWjfa} in the operator space setting.

\begin{thm}\label{H thm}
Let $\phi \in SL\big( \wh{D'},\Dhat; B(\YVil)\big)$ and let
$k\in\Procphiwr(\YVil\to\YVil:\againDdense, \Ddense)$ be a
$(\Ddense, \againDdense)$-weak solution of the quantum stochastic
differential equation~\eqref{id QSDE}.
Then $k$ is a Markov-regular quantum stochastic cocycle and the generators
of its associated semigroups are given by
\begin{equation}\label{psi and phi}
\psi_{c',c}=\phi^{\wh{c'}}_{\wh{c}} \qquad (c'\in D', c\in D).
\end{equation}
\end{thm}
\begin{proof}
Let $\xi' = v' \ve (g') \in \againDdense , \xi \in v \ve (g) \in \Ddense$
and $t \in \Rplus$. Define $l^{g',g}_t \in B(\YVil)$ by~\eqref{lfg}
where $P^{c',c}$ is the norm continuous semigroup in $B(\YVil)$ with
generator $\phi^{\wh{c'}}_{\chat}$.
Then $m^{g',g}_t := k^{g',g}_t - l^{g',g}_t$ satisfies
\[
\big\la v' , m^{g',g}_t (x) v \big\ra = \int^t_0 ds \,
\big\la v' ,m^{g',g}_s (\phi^{\wh{g'}(s)}_{\wh{g} (s)} (x) ) v \big\ra .
\]
Iterating this gives
\begin{align*}
\lefteqn{ \big\la v' , m^{g',g}_t (x) v \big\ra} \\
& \quad = \int^t_0 ds_n \cdots \int^{s_2}_0 ds_1 (\omega_{\xi',\xi} \circ k_{s_1}
   - \omega_{v',v} \circ l^{g',g}_{s_1} )
   ( \phi^{\wh{g'} (s_1)}_{\ghat (s_1)} \circ \cdots \circ
\phi^{\wh{g'}(s_n)}_{\ghat (s_n)} ) (x) .
\end{align*}
By $\phi$-weak regularity of $k$ and norm continuity of $l^{g',g}$, the integrand has a bound
of the form $C \| x\| M^n$ where the constants $C$ and $M$ are independent of $n$. The identity
$k^{g',g}_t = l^{g',g}_t$ follows and so, by Lemma~\ref{sgp decomp lem},
$k$ is a quantum stochastic cocycle
with associated semigroups $\{ P^{c',c} : c' \in D' , c \in D\}$. This
completes the proof.
\end{proof}

It follows from~\eqref{psi and phi} that the associated semigroups are cb-norm continuous if
and only if the sesquilinear map $\phi$ is $CB (\YVil)$-valued.

\begin{rems}
Note that, in this case, the `form representation' of $k$~\eqref{form 
repn} is given by:
\[
k^{g',g}_s =
\int_{\Gamma_{[0,s]}} \, d \sigma \, \up^{g',g}_\sigma
\]
where
$\up^{g',g}_\sigma = \id_{\YVil}$ when $\sigma = \emptyset$ and
\[
\up^{g',g}_\sigma =
\phi^{\wh{g'}(s_1)}_{\wh{g}(s_1)}\circ \cdots \circ
\phi^{\wh{g'}(s_n)}_{\wh{g}(s_n)}
\text{ for } \sigma = \{s_1 < \cdots < s_n\}.
\]
In particular, if $k=k^\phi$ where
$\phi\in L\big(\Dhat; \kB(\YVil;\VMotk) \big)$ then
\[
\up^{g',g}_\sigma =
\omega_{\pi_{\wh{g'}}(\sigma),\pi_{\ghat}(\sigma)}
\fullcomp \up_{\# \sigma},
\]
where $\up = \up^\phi$ is defined by~\eqref{upsilon} with
$\kappa = \id_V$,
and the cocycle relation may be expressed as follows:
\[
\int_{\Gamma_{[0,r+t]}}\! d\sigma \, \up^{g',g}_\sigma =
\int_{\Gamma_{[0,r]}}\! d\rho
\int_{\Gamma_{[0,t]}}\! d\tau \, \up^{g',g}_\rho \circ \up^{S^*_r g',S^*_r g}_\tau .
\]
In this case the associated semigroup generators are given by
\begin{equation}\label{psi omega phi}
\psi_{c',c} = \omega_{\wh{c'},\wh{c}}\fullcomp \phi.
\end{equation}
\end{rems}

\begin{cor} \label{H cor}
Let $\phi \in L\big( \YVil;\YVil\ulot \Op (\Dhat ) \big)$ and suppose that
$\YVil$ is finite dimensional. Then $k^\phi$ is an
$L\big( \YVil;\YVil \ulot \Op(\ExpsD)\big)$-valued Markov-regular quantum
stochastic cocycle.
\end{cor}

\begin{proof}
This follows from the theorem above and Theorem \ref{existence}
since, for finite dimensional $\YVil$, there are
natural linear identifications
\[
L\big( \YVil;\YVil \ulot \Op(E)\big) =
L\big( E ; L(\YVil;\YVil \ulot | \Hil \ra ) \big) =
L \big( E ; CB(\YVil;\YVil\otM |\Hil\ra ) \big) ,
\]
for $(E,\Hil)$ equal in turn to $(\Dhat,\kilhat)$ and $(\ExpsD,\FFock)$.
\end{proof}

We now begin to develop converse results. The first is a coordinate-free
counterpart to Theorem 5.6 of~\cite{LWjfa} in the operator space setting.

\begin{thm}
\label{cocycle thm}
Let $k \in \QSCddagger (\YVil : \Ddense, \againDdense)$ and suppose that
$k$ is Markov-regular and the maps
$t \mapsto k_t (x) \xi$ and $t \mapsto k_t (x)^* \xi'$
\tu{(}$x \in \YVil, \xi \in \Ddense, \xi' \in \againDdense$\tu{)}
are all continuous at $0$.
Then $k$ is a $(\againDdense,\Ddense)$-weak solution of the quantum
stochastic differential equation~\eqref{id QSDE} for some
$\phi \in SL\big(\wh{D'},\wh{D}; B(\YVil)\big)$.
\end{thm}

\begin{proof}
Define a map as follows
\[
\phi : \wh{D'} \times \Dhat \to B(\YVil),
\left( \binom{z'}{c'} , \binom{z}{c} \right) \mapsto
\begin{bmatrix} \ol{z'-1} & 1 \end{bmatrix}
\begin{bmatrix} \psi_{0,0} & \psi_{0,c} \\ \psi_{c',0} & \psi_{c',c}
\end{bmatrix}
\begin{bmatrix} z-1 \\ 1 \end{bmatrix}
\]
where \{$\psi_{c',c}:c'\in D', c\in D$\} are the generators of $k$'s 
associated semigroups and, for $x \in \YVil$,
let $\phi (x)$ denote the corresponding map $\wh{D'} \times \wh{D} \to 
\YVil$.
 Markov-regularity implies that $l^{g',g}$, given by~\eqref{lfg},
satisfies
\[
l^{g',g}_t = \id_{\YVil} + \int_0^t ds \, l^{g',g}_s \circ \psi_{c',c},
\]
where $c'=g'(t_-)$ and $c=g(t_-)$.
But, by the semigroup decomposition, $l^{g',g}=k^{g',g}$; since
$\phi^{\wh{c'}}_{\chat} = \psi_{c',c}$ it therefore suffices only to prove
that $\phi$ is sesquilinear.

Accordingly, fix $v' \in \init' , v \in \init $ and $ x \in \YVil$ and 
note the
identity
\[
\big\la v', \phi^{\zeta'}_{\zeta}(x) v\big\ra =
\lim_{t \to 0^+} t^{-1}
\big\la \alpha (t) , \beta (t) \big\ra
\]
where $\zeta' = \binom{z'}{c'}\in\wh{D'}$, $\zeta = \binom zc \in\Dhat$,
\begin{align*}
\alpha (t) &= \big( k^\dagger_t (x^*) - x^* \ot 1 \big)
\Big( v' \ot \big\{ (z'-1) \ve (0) + \ve (c'_{[0,t[}) \big\} \Big)
\text{ and} \\
\beta (t) &= v \ot \big( z,c_{[0,t[} , (2!)^{-1/2} (c_{[0,t[})^{ \ot 2} ,
\ldots \big) ,
\end{align*}
Thus if $\zeta = \zeta_1 +\lambda \zeta_2$
for $\zeta_i = \binom{z_i}{c_i} \in \Dhat$ $(i=1,2)$ and $\lambda \in
\Comp$ then
\[
\big\la v', \big(
\phi^{\zeta'}_\zeta(x) -
\phi^{\zeta'}_{\zeta_1}(x)
- \lambda
\phi^{\zeta'}_{\zeta_2}(x)
\big) v \big\ra
= \lim_{t \to 0^+} \big\la \alpha (t) , \gamma (t) \big\ra
\]
where
\[
\gamma (t) =
t^{-1} v\ot \big( (n!)^{-1/2} \big\{ c^{\ot n} -(c_1)^{\ot n}
- (\lambda c_2)^{\ot n} \big\}
\ot 1_{[0,t[^n} \big)_{n \geq 2} .
\]
Since $\gamma$ is locally bounded and $\alpha (t) \to 0$ as $t \to 0$, by
the continuity of the process $k^\dagger$, this shows that $\phi(x)$ is
linear in its second argument. A very similar argument, in which
the roles of $k$ and $k^\dagger$ are exchanged, shows that $\phi(x)$ is
conjugate linear in its first argument. The result follows.
\end{proof}

\begin{rems}
In view of Corollary~\ref{uniqueness cor}, $k$ is the \emph{unique} linear
$(\againDdense, \Ddense)$-weak solution of~\eqref{id QSDE}.
In particular, if either
\begin{alist}
\item
$\phi\in L\big(\Dhat; \kB(\YVil;\YVMotk \big)$, or
\item
$\YVil$ is finite dimensional and
$\phi\in L\big( \YVil;\YVil\ulot \Op (\Dhat ) \big)$,
\end{alist}
then $k=k^\phi$ and so
satisfies the equation \emph{strongly}.
If $\YVil$ is a $C^*$-algebra and $k$ is completely positive and 
contractive
then (a) holds (by \cite{LWjfa}, Theorem 5.4 and \cite{LWblms}, Theorem 2.4);
it also holds if $\noise$ is finite dimensional.
\end{rems}

We next identify a necessary and sufficient condition for (b) to hold. To
this end let $\QSCHc (\YVil : \Ddense,\againDdense)$
denote the collection of
cocycles $k\in\QSC (\YVil : \Ddense,\againDdense)$ for which
\begin{equation} \label{Holder}
k_{t,|\ve\ra}(x) \text{ is bounded and }
s\mapsto k_{s,|\ve\ra}(x)\in\VMotF
\text{ is H\"older } \tfrac{1}{2}\text{-continuous at } 0
\end{equation}
$(t\in\Rplus, \ve\in\ExpsD, x\in\YVil)$. Let
$\QSCddagger(\YVil : \Ddense,\againDdense)$ denote the set of processes
$k\in\Procdagger(V:\Ddense,\againDdense)$ such that both
$k$ and $k^{\dagger}$ satisfy~\eqref{Holder}.

\begin{lemma}\label{B lem}
Let $k\in\QSCddaggerHc(\YVil:\Ddense,\againDdense)$ be Markov-regular,
with resulting $\phi$ \tu{(}from Theorem~\ref{cocycle thm}\tu{)} viewed
as a linear map $\YVil\to SL(\wh{D'},\wh{D};\YVil)$.
Then, for all $x\in\YVil$, $\phi(x)$ is separately continuous in each
argument.
\end{lemma}
\begin{proof}
Fix $x \in \YVil$ and let $\zeta' = \binom{z'}{c'} \in \wh{D'}$ and
$\zeta = \binom zc \in \Dhat $.
Then, in terms of the generators of the associated semigroups,
$\phi^{\zeta'}_\zeta(x)$ equals
\[
\ol{z'} \big\{ (z-1) \psi_{0,0} (x) + \psi_{0,c} (x) \big\} +
(z-1)\big\{ \psi_{c',0} (x) - \psi_{0,0} (x) \big\} +
\big\{ \psi_{c',c} (x) - \psi_{0,c} (x) \big\}
\]
and, for each $v' \in \init' , e \in D$ and $v \in \init$, setting $C(x,e)
=
\sup \big\{ t^{-1/2} \| k_{t,| \ve \ra} (x) -
x \ot | \ve \ra \| :t \in ]0,1[ \big\}$
where $\ve = \ve (e_{[0,1[})$,
\begin{align*}
\lefteqn{ \big|
\big\la v', (\psi_{c',e} (x) - \psi_{0,e} (x) ) v \big\ra
\big| } \\
 & \quad = \lim_{t \to 0^+} t^{-1} e^{-t \la c',e \ra}
     \Big| \big\la v' \ot \{ \ve (c'_{[0,t[}) - \ve (0) \big\} ,
     \big( k_t (x) - x \ot 1 \big) v \ot \ve (e_{[0,1[}) \big\ra \Big| \\
 & \quad \leq \| v' \| \, \| c' \| C(x,e) \| v \| .
\end{align*}
Thus $\| \psi_{c',e} (x) - \psi_{0,e} (x) \| \leq \| c' \| C(x,e)$. It
follows that
\begin{align*}
\lefteqn{ \| \phi^{\zeta'}_\zeta(x) \|} \\
& \quad \leq |z'| \big\| (z-1) \psi_{0,0} (x) + \psi_{0,c} (x) \big\|
     + | z-1 | \, \| c'\| C(x,0) + \| c'\| C(x,c) \\
& \quad \leq \| \zeta' \| M(\zeta , x) ,
\end{align*}
where $M(\zeta ,x)$ is a constant independent of $\zeta'$. Thus the
sesquilinear map $\phi(x)$ is continuous in its first argument. Again
applying the above argument to $k^\dagger$ yields continuity in the second
argument.
\end{proof}

\begin{rem}
If $\YVil$ is finite dimensional then the continuity assumption
introduced in
\eqref{Holder} is equivalent to H\"older-continuity at $0$ of the map
\[
s\mapsto k_{s,|\ve\ra}\in B(\YVil;\YVMotF)
\quad
(\ve\in\ExpsD).
\]
If $\init$ is finite dimensional then
this further reduces to the pointwise strong continuity condition
\[
s\mapsto k_s(x) \xi \in \init' \ot \FFock
\text{ is H\"older } \tfrac{1}{2}\text{-continuous at } 0
\quad
(x \in \YVil, \xi \in \Ddense).
\]
We alert the reader to the fact that not all
finite dimensional operator spaces can be concretely
realised in $B(\Hil)$,
in the sense of a completely isometric embedding,
for a finite dimensional Hilbert space $\Hil$.
For more on this point, and for details of
an example given by the operator space spanned by the canonical unitary
generators of the universal $C^*$-algebra of a free group
$\mathbb{F}_n$ ($n \geq 3$), we refer to~\cite{Pisier}.
\end{rem}

\begin{thm}
\label{f.d. cocycle propn}
Let $k \in \QSCddaggerHc (\YVil : \Ddense, \againDdense)$ and suppose that
$\YVil$ is finite dimensional. Then there is
$\phi \in L \big( \YVil;\YVil \ulot \Opdagger (\Dhat , \wh{D'}) \big)$ 
such
that $k = k^\phi$.
\end{thm}

\begin{proof}
Note first that, since $\YVil$ is finite dimensional, the continuity
assumption implies that $k$ is Markov-regular. Let
$\phi\in L(\YVil; SL(\wh{D'},\wh{D};\YVil))$ be the map resulting from
Theorem~\ref{cocycle thm}.
Choose an ordered basis $\{ x_1, \ldots , x_n \}$ of $\YVil$ and for
$x \in \YVil$, $\zeta' \in \wh{D'}$
and $\zeta \in \Dhat$, let
$\phi^{\zeta'}_\zeta(x)^i$,
$i = 1, \ldots , n$, denote the components of $\phi^{\zeta'}_\zeta(x)$,
with respect to this basis. By Lemma~\ref{B lem} each functional
$\phi(x)^i : \wh{D'} \times \Dhat \to \Comp$ is sesquilinear and
continuous in each argument; it is therefore
given by an operator $\phi^{(i)}(x) \in \Opdagger (\Dhat , \wh{D'})$:
\[
\phi^{\zeta'}_\zeta(x)^i = \la \zeta' , \phi^{(i)}(x) \zeta \ra
\qquad (\zeta'\in\wh{D'}, \zeta\in\Dhat).
\]
Moreover, each map $x \mapsto \phi^{(i)}(x)$ is clearly linear.
Thus, setting
\[
\phi (x) = \sum^n_{i=1} x_i \ot \phi^{(i)}(x)
\]
defines a linear map
$\phi : \YVil \to \YVil \ulot \Opdagger (\Dhat , \wh{D'})$.
Therefore, by Corollary~\ref{H cor}, $\phi$ generates a stochastic
cocycle.
In view of the identity
\[
(\omega_{\wh{c'} , \chat} \fullcomp \phi) (x) =
\sum^n_{i=1} \phi^{\wh{c'}}_{\chat}(x)^i x_i =
\phi^{\wh{c'}}_{\chat}(x) = \psi_{c',c}(x)
\]
and Theorem~\ref{H thm}, $k$ has the same associated semigroups as
the cocycle $k^\phi$. Thus $k = k^\phi$ and the proof is complete.
\end{proof}

By \emph{finite localisability} for a process
$k\in\Proc(\YVil\to\YVil: \Ddense,\againDdense)$ we mean
finite localisability for each $k_t$.
Combining the above result with Corollary~\ref{uniqueness cor} and
Theorem \ref{more local existence}, straightforward
localisation arguments allow us to summarize the new results of this 
section
as follows.

\begin{cor}  \quad
\begin{alist}
\item \label{QGC a}
Let $\phi \in L\big( \YVil ; \YVil \ulot \Op (\Dhat)\big)$ be finitely
localisable. Then
$k^\phi \in \mathbb{QSC}_{\mathrm{Hc}}(\YVil : \Ddense , \againDdense)$
and is finitely localisable, moreover if 
$\phi\in L\big( \YVil ; \YVil \ulot \Op^{\ddagger} (\Dhat,\wh{D'})\big)$
then
$k^\phi \in \mathbb{QSC}^\ddagger_{\mathrm{Hc}}(\YVil : \Ddense , 
\againDdense)$.

\item \label{QGC b}
Conversely, let
$k \in \mathbb{QSC}^\ddagger_{\mathrm{Hc}}(\YVil :\Ddense , \againDdense)$
be finitely localisable. Then there is a finitely localisable map
$\phi\in L\big( \YVil ; \YVil \ulot \Op^{\ddagger} (\Dhat,\wh{D'})\big)$
such that $k = k^\phi$.
\end{alist}
\end{cor}

\subsection{Application to coalgebraic cocycles}
Theorem \ref{f.d. cocycle propn} yields an alternative proof of the
principal implication in Theorem 5.8 of \cite{LSaihp} which states
that if $\Clg$ is a coalgebra with coproduct $\Delta$ and counit $\counit$,
then any H\"older-continuous quantum stochastic \emph{convolution cocycle}
$l\in \Procdagger (\Clg \to \Comp;\ExpsD,\Exps_{D'})$,
with H\"older-continuous conjugate, satisfies a coalgebraic quantum stochastic
differential equation
\begin{equation}\label{coalg QSDE}
dl_t = l_t\star_{\tau} d\Lambda_{\varphi}(t),  \quad
l_0 = \iota\circ\counit,
\end{equation}
for some map $\varphi\in L(\Clg;\Op^{\ddagger}(\Dhat,\wh{D'}))$. We end 
with a sketch of a proof of this. The Fundamental Theorem on Coalgebras 
and localisation arguments allow us to effectively assume that $\Clg$ is 
finite dimensional.
Assuming this, linearly embed $\Clg$ into $B(\init)$, for some
(finite dimensional) Hilbert space $\init$,
and observe that the process $k\in
\Procdagger(\Clg \to \Clg;\init\ulot\ExpsD,\init\ulot\Exps_{D'})$,
defined by the formula
\begin{equation}
k_t = (\id_{\Clg} \ulot l_t) \circ \Delta
\quad
(t\geq 0),
\label{conv1}
\end{equation}
is a H\"older-continuous quantum stochastic cocycle on $\Clg$.
Theorem \ref{f.d. cocycle propn} then implies that $k$
satisfies the quantum stochastic differential equation~\eqref{id QSDE} for
some
$\phi \in L(\Clg; \Clg \ulot \Op^{\ddagger}(\Dhat,\wh{D'}))$. Set
\begin{equation}
\label{conv2}
\varphi = (\epsilon \ulot \id_{\Op^{\ddagger}(\Dhat,\wh{D'})} )\circ \phi.
\end{equation}
It is then easily checked that the convolution cocycle $l$ satisfies the
coalgebraic quantum stochastic differential equation~\eqref{coalg QSDE}.

\begin{rem}
The idea outlined here, of using correspondences such as \eqref{conv1}
and~\eqref{conv2} for moving between
quantum stochastic cocycles and quantum stochastic convolution cocycles,
or their respective stochastic generators, also works well in
the analytic context of quantum stochastic convolution
cocycles on operator space coalgebras.
This enables application of known results for quantum stochastic cocycles
to the development of a theory of quantum L\'evy processes on compact
quantum groups and the characterisation of their stochastic generators.
This is done in the forthcoming paper \cite{LSqscc2} which also contains 
many examples.
Dilation of completely positive convolution cocycles on 
a $C^*$-bialgebra to $*$-homomorphic convolution cocycles is treated in
\cite{Sdilations}. The main results, in both the algebraic and 
$C^*$-algebraic cases, are summarized in~\cite{LSbedlewo}.
\end{rem}

\end{document}